\documentclass[preprint,3p]{elsarticle}
\usepackage[latin1]{inputenc}
\usepackage{amsfonts}
\usepackage{enumitem}
\usepackage{color}
\usepackage{bm}
\usepackage[T1]{fontenc}
\usepackage[french,english]{babel}
\usepackage{graphicx,dsfont}
\usepackage{amssymb,amsmath,amsthm,}

\newtheorem{define}{Definition}[section]
\newtheorem{pro}{Proposition}[section]
\newtheorem{Lem}{Lemma}[section]
\newtheorem{cor}{Corollary}[section]

\newtheorem{remark}{Remark}[section]
\theoremstyle{plain} \newtheorem{thm}{Theorem}[section]

\DeclareMathOperator{\dist}{dist}
\DeclareMathOperator{\dom}{dom}

\DeclareMathOperator{\supp}{supp}
\DeclareMathOperator{\diam}{diam}

\catcode`\@=11
\makeatletter

\@addtoreset{equation}{section}

\newenvironment{abstracts}
 {\global\setbox\absbox=\vbox\bgroup
    \hsize=\textwidth
    \linespread{1}\selectfont}
 {\vspace{-\bigskipamount}\egroup}
\renewenvironment{abstract}[1][]
 {\if\relax\detokenize{#1}\relax\else\selectlanguage{#1}\fi
  \noindent\textbf{\abstractname}\par\medskip\noindent\ignorespaces}
 {\par\bigskip}

\begin{document}

\title{Combined effects in mixed local-nonlocal stationary problems}
\author[1]{Rakesh Arora\fnref{fn1}}
 \ead{arora.npde@gmail.com, arora@math.muni.cz}

\address[1]{Department of Mathematics and Statistics, Masaryk University, Building 08,
Kotl\'{a}\v{r}sk\'{a} 2, Brno, 611 37, Czech Republic}

\author[2,3]{Vicen\c tiu D. R\u adulescu}
 \ead{radulescu@inf.ucv.ro}

\address[2]{Faculty of Applied Mathematics, AGH University of Science and Technology, al. Mickiewicza
30, 30-059 Krak\'ow, Poland}
\address[3]{Department of Mathematics, University of Craiova, Street A.I. Cuza 13, 200585 Craiova, Romania}

\begin{abstracts}

\begin{abstract}
In this work, we study an elliptic problem involving an operator of mixed order with both local and nonlocal aspects, and in either the presence or the absence of a singular nonlinearity. We investigate existence or non-existence properties, power and exponential type Sobolev regularity results, and the boundary behavior of the weak solution, in the light of the interplay between the summability of the datum and the power exponent in singular nonlinearities. 
\end{abstract}

\end{abstracts}


\maketitle

\tableofcontents

\section{Introduction}

In this article, we study the fine properties of the weak solution to an elliptic problem involving a mixed type operator $\mathcal{L}$, given by
\begin{equation}\label{mixed:oper}
    \mathcal{L}:= (-\Delta) + (-\Delta)^s \quad \text{for} \ 0<s<1. 
\end{equation}

Here, the word ``mixed'' refers to the type of the operator combining both local and nonlocal features, and to the differential order of the operator. The operator $\mathcal{L}$ is obtained by superposition of the classical Laplacian $(-\Delta)$ and the fractional Laplacian  $(-\Delta)^s$, for a fixed parameter $s \in (0,1)$, defined as
$$(-\Delta)^s u = C(N,s) P.V. \int_{\mathbb{R}^N} \frac{u(x)-u(y)}{|x-y|^{N+2s}} ~dy.$$
The term $``P.V."$ stands for Cauchy's principal value and $C(N,s)$ is a normalizing constant, whose explicit expression is given by
$$C(N,s)= \left(\int_{\mathbb{R}^N} \frac{1-cos(\xi_1)}{|\xi|^{N+2s}} d\xi\right)^{-1}.$$
The above choice of the constant $C(N,s)$ arises from the equivalent definition of $(-\Delta)^s$ to view it as a pseudo-differential operator of symbol $|\xi|^{2s}.$ Without going into the details of the appearance of such type of nonlocal operator in real-world phenomenons and motivation behind studying problems involving such nonlocal operators, we refer the reader to review the famous Hitchhiker's guide \cite{Nezza_palatucci_valdi} and reference within. 

The mixed operators of the form $\mathcal{L}$ in \eqref{mixed:oper} appears naturally in applied sciences, to study the role of the impact caused by a local and a non-local change in a physical phenomenon. More precisely, it can be understood through the following biological scenario where the population with density $u$ can possibly alternate both short and long-range random walks (namely, a classical random walk and a L\'evy flight), and this could be
driven, for example, by a superposition between local exploration of the environment and hunting strategies (for a through discussion see, \cite{KLS} for mixed dispersal movement strategy, \cite{MV_2017} for nonlocal diffusion strategy, \cite{CHL} for conditional dispersal strategy). These type of operators also arises in the models obtained from superposition of two different scaled stochastic processes. For a detailed presentation on this, we refer the reader to \cite{MRS}.

Very recently a great amount of attention has been paid in studying elliptic problems involving mixed type of operator having both local and nonlocal behaviours. Some questions related to structural results like existence, maximum principle and interior Sobolev and Lipschitz regularity \cite{Aba_cozzi, Biagi_Valdi, Biagi_Vecchi}, symmetry results \cite{Biagi_Valdi-2}, Faber-Krahn type inequality \cite{Biagi_Valdi-1}, Neumann problems \cite{MRS}, Green functions estimates \cite{Chenkim_Green1, Chenkim_Green2} has been answered. 

The study of elliptic or integral equations involving singular terms started in the early sixties by the works of Fulks and Maybee \cite{Fulks.al_1960}, originating from the models of steady-state temperature distribution in an electrically conducting medium.  On the one hand, the study of such types of equations is a challenging mathematical problem. On the other hand, they appear in a variety of real-world models. To demonstrate an application, let us consider $\Omega$ be an electrically conducting medium in $\mathbb{R}^3$ where the local voltage drop is described by the function $f$ and $u$ be the steady state temperature distribution in the region $\Omega$. Then, if $\sigma(u)$ is the electrical resistivity which is, in general, a function
of the temperature u, in particular $\sigma(u) =u^{\gamma}$, the rate of generation of heat at any point $x$ in the medium is $\frac{f(x)}{u^\gamma}$ and the temperature distribution in the conducting medium satisfies the local counterpart of the equation (see below \eqref{main:problem}). For interested readers, we refer to \cite{Nachman.al_1986,  Stuart_1974, Diaz.al_1987} for applications in the pseudo-plastic fluids, Chandrasekhar equations in radiative transfer, and in non-Newtonian fluid flows in porous media and heterogeneous catalysts. \vspace{0.2cm}\\
Motivating from the above discussion, we study the following mixed local/nonlocal elliptic problem in the presence of weight function $f$ and singular nonlinearities
\begin{equation}\label{main:problem}
    \mathcal{L} u = \frac{f(x)}{u^\gamma}, \quad u>0 \quad \text{ in }  \Omega,
\end{equation}
subject to the homogeneous Dirichlet boundary conditions
\begin{equation}\label{bdry:cond}
    u  = 0 \ \text{ in } \ \mathbb{R}^N \setminus \Omega,
\end{equation}
where $\Omega \subset \mathbb{R}^N$, $N \geq 2$, $\gamma \geq 0$. The function $f: \Omega \to \mathbb{R}^+$ either belongs to the Lebesgue class of functions $L^r(\Omega)$ for some $1 \leq r \leq \infty$ or has a growth of negative powers of distance function $\delta$ near the boundary {\it i.e.} $f(x) \sim \delta^{-\zeta}(x)$ for some $\zeta \geq 0$ and $x$ lies near the boundary $\partial \Omega$.

\subsection{Understanding the notion of a solution}
We start by understanding the meaning of a ``weak'' solution for the problem \eqref{main:problem}-\eqref{bdry:cond}. An elementary way to define the notion of a solution of the problem \eqref{main:problem} is given by: a function $u$ such that
\begin{enumerate}
    \item[(i)] $u>0$ a.e. in $\Omega$,
    \item[(ii)] the following weak formulation equality holds:
    \begin{equation}\label{main:weak:formula}
\begin{split}
    \int_{\Omega} \nabla u \cdot \nabla \psi ~dx + \frac{C(N,s)}{2} \int_{\mathbb{R}^N} \int_{\mathbb{R}^N} \frac{(u(x)- u(y))(\psi(x)- \psi(y))}{|x-y|^{N+2s}} ~dx ~dy = \int_{\Omega}  \frac{f(x)}{u^\gamma}\psi ~dx.
\end{split}
\end{equation}
for a class of test functions $\psi \in \mathcal{T}(\Omega)$ and a function $u$ ``regular'' enough so that all the integrals are well defined.
\end{enumerate}
The first condition $(i)$ is imposed to give a meaning to the term $u^{-\gamma}$ known as ``singular nonlinearities'' and the second condition $(ii)$ is motivated and obtained by multiplying a smooth function $\psi$ to the equation \eqref{main:problem} and using standard integration by parts formula for a smooth function $u$. Since solutions to equations involving a fractional Laplacian and singular nonlinearities generally are not of class $C^2$ therefore a
solution to \eqref{main:problem} has to be understood in the ``weak'' sense via \eqref{main:weak:formula}.

Assuming $\mathcal{T}(\Omega)$ to be a class of smooth test functions (for e.g. $C_c^\infty(\Omega)$), a $``$natural$"$ space to look for the solution $u$ of the problem \eqref{main:problem}-\eqref{bdry:cond}, more accurately, to define the integrals on the left hand side of \eqref{main:weak:formula}, is the following:
$$\mathbb{H}(\Omega):= \{u \in H^1(\mathbb{R}^N): \ u=0 \ \text{in}\ \mathbb{R}^N \setminus \Omega\}.$$

In the light of the boundary regularity of $\Omega$, it is well known that that the space $\mathbb{H}(\Omega)$ can be identified with $H_0^1(\Omega).$ Precisely, we understand the identification (see \cite[Proposition 9.18]{brezis}) in the following way: the function $u \in H_0^1(\Omega)$ as a zero extension of the function $\tilde{u}:= u \cdot {\bf 1}_{\Omega} \in \mathbb{H}(\Omega)$ and
$$ u \in  \mathbb{H}(\Omega) \Longrightarrow u|_\Omega \in H_0^1(\Omega).$$

In view of above identification and employing the classical embedding theorem \cite[Proposition 2.2]{Nezza_palatucci_valdi}, both the integrals on left hand side of \eqref{main:weak:formula} are well defined in $H_0^1(\Omega).$ Due to the fact that $\gamma \geq 0$, the nonlinearity $u^{-\gamma}$ in the our problem \eqref{main:problem}-\eqref{bdry:cond} may blow up near the boundary and this is the reason why we regard \eqref{main:problem} as an equation with ``singular nonlinearities''. By taking into account the singular nature of the nonlinearities and the regularity of the datum $f$, specifically, when the nature of the singularity is strong {\it i.e.} $\gamma \gg 1$ (see \cite{Lazer_Mckenna}) or $f \in L^r(\Omega)$ when $r$ is close to 1 (see \cite{Bocc_ors}), we cannot always expect our solution $u$ in $H_0^1(\Omega)$ and in place of that, either we have $u^\alpha \in H_0^1(\Omega)$ for some $\alpha \geq 1$ or $u \in W_0^{1,q}(\Omega)$ for some $q \in [1,2)$. For these reasons, as customary in the literature, we adopt the
following definition to understand the Dirichlet datum in a generalized sense (see \cite{Canino_et_al, Arora_et-al-1}):
\begin{define}\label{def:boundary}
A function $u \leq 0$ on $\partial \Omega$, if $u=0$ in $\mathbb{R}^N \setminus \Omega$ and for any $\epsilon >0$, we have 
$$(u-\epsilon)^+ \in H_0^1(\Omega).$$
We say that $u=0$ on $\partial \Omega$, if $u \geq 0$ and $u \leq 0$ on $\partial \Omega.$
\end{define}

Now, to give a meaning to the integral present on the right hand side of \eqref{weak:formula4}, we choose a suitable class of test functions $\mathcal{T}(\Omega)$ depending upon the value of the exponent $\gamma \geq 0$ and the regularity of the datum $f$ and the solution $u.$ Motivating from the above discussion, the exact notion of weak solution to our main problem \eqref{main:problem}-\eqref{bdry:cond} with different choices of a class of test functions and sufficient regularity of the solution $u$, is detailed in Section \ref{main:res} while stating the main results.

\subsection{Previous work}
One of the seminal breakthroughs in the study of singular nonlinearities was the work of Crandall, Rabinowitz, and Tartar \cite{Crandall.al_1977} which majorly setup this direction of research. Afterward, a large number of publications has been devoted to investigate a diverse spectrum of issues rotating around local/nonlocal elliptic equations involving the singular nonlinearities (see for example \cite{Diaz.al_1987, Fulks.al_1960, Stuart_1974} and monographs \cite{Ghergu.al_2008, Hernandez.al_2006}).  Let us recall some known results in the literature for both local and nonlocal elliptic equations with singular nonlinearities. \vspace{0.1cm}

In the local case, Crandall et al. \cite{Crandall.al_1977} studied the singular boundary value problem \eqref{main:problem}-\eqref{bdry:cond} with Laplace operator and $f=1$. By using the classical method of sub-supersolutions on the non-singular approximating problem, they proved the existence and uniqueness results of the classical solution of our original problem. In addition, by exploiting the second-order ODE techniques and localization near the boundary, the boundary behavior of the solution is deduced. By Stuart \cite{Stuart_1976}, similar results on the existence of solutions were obtained using, this time, an approximation argument with respect to the boundary condition. Actually, both papers \cite{Crandall.al_1977, Stuart_1976} provide results for more general differential operators with smooth coefficients, not necessarily in divergence form, and for non-monotone
nonlinearities as well. The same model of elliptic equations with singular nonlinearities and $f \in C^\alpha(\overline{\Omega})$, was considered by Lazer and McKenna \cite{Lazer_Mckenna} in which they simplified the proof of boundary behavior of classical by constructing appropriate sub and super solutions. In addition to that, they also obtained the optimal power related the the existence of finite energy solutions. In fact, a solution $u \in H_0^1(\Omega)$ exists if and only if $\gamma<3.$ In \cite{Yijing_Zhang}, Yijing and Zhang analyzed the threshold value $3$ when the datum $f \in L^1(\Omega)$ and a positive function, and provide a classical Lazer-Mckenna obstruction. In \cite{Bocc_ors, Cam_et_al}, authors studied the existence and uniqueness results when $f \in L^r(\Omega)$ for $r \geq 1$ and showed how the regularity of this solution depends upon the summability of the datum and the singular datum. In particular, Boccardo and Orsina \cite{Bocc_ors} proved the existence and regularity of distributional solution: 
\begin{equation*}
	\begin{aligned}
	\left\{
	\begin{array}{ll}
		u \in W_0^{1,\frac{Nr(1+\gamma)}{N-r(1-\gamma)}}(\Omega)  & \text{ if } \ 0<\gamma <1 \ \text{and} \ f \in L^r(\Omega) \ \text{with} \ r \in [1,(2^\ast/(1-\gamma))'),\vspace{0.1cm}\\
		u \in H_0^{1}(\Omega) & \text{ if } \ 0<\gamma <1 \ \text{and} \ f \in L^r(\Omega) \ \text{with} \ r=(2^\ast/(1-\gamma))',\vspace{0.1cm}\\
		u \in H_0^{1}(\Omega) & \text{ if } \ \gamma =1 \ \text{and} \ f \in L^1(\Omega),\vspace{0.1cm}\\
	u^{\frac{1+\gamma}{2}} \in H_0^{1}(\Omega) & \text{ if } \ \gamma >1 \ \text{and} \ f \in L^1(\Omega),
	\end{array} 
	\right.
	\end{aligned}
	\end{equation*}
Extending the work of \cite{Bocc_ors}, Arcoya and M\'erida in \cite{Arcoya_merida}, studied some particular cases of strongly singular elliptic equations {\it i.e.} $1< \gamma< \frac{3r-1}{r+1}$ and $f \in L^r(\Omega)$, $r>1$ and $f$ is strictly far away from zero on $\Omega$ and proved the power type Sobolev regularity  
\[ u^\alpha \in H_0^1(\Omega) \ \text{for all} \ \alpha \in \left(\frac{(r+1)(1+\gamma)}{4r}, \frac{1+\gamma}{2}\right] \]
Moreover, in connection with the same problem, in \cite{Bocc_diaz}, Boccardo and D\'iaz proved the uniqueness of finite energy solution by extending the set of admissible test functions. For similar works concerning the local elliptic or integral equations with purely singular nonlinearities, we refer to \cite{Zheng.al_2004, Diaz.al_1987, Gomes_1986, Stuart_1974, brando_chia} and singular nonlinearities with source terms or absorption terms, we refer to \cite{Coc.al_1989, Stuart_1976, Can.al_2004, Oliva_petitta} with no intent to furnish an exhaustive list. \vspace{0.1cm}

Turning to the nonlocal case, the singular problems have been investigated more recently and there are few works in the literature, in particular, with the fractional Laplacian $(-\Delta)^s$ and related to Lazer-Mckeena type problem (see for instance \cite{Adi.al_2018, Arora.al_2020, Arora_et-al-1, Barrios.al_2015}). In \cite{Barrios.al_2015}, Barrios et al. studied the solvability of the nonlocal problem in presence of singular nonlinearities and weight function $f$. In particular, they proved the existence and regularity of solution in very weak sense depending upon the regularity of the datum $f$ and the singular exponent $\gamma$,
\begin{equation*}
	\begin{aligned}
	\left\{
	\begin{array}{ll}
		u \in H_0^{s}(\Omega) & \text{ if } \ 0<\gamma \leq 1 \ \text{and} \ f \in L^r(\Omega) \ \text{with} \ r=(2_s^\ast/(1-\gamma))',\vspace{0.1cm}\\
	u^{\frac{1+\gamma}{2}} \in H_0^{s}(\Omega) & \text{ if } \ \gamma >1 \ \text{and} \ f \in L^1(\Omega),
	\end{array} 
	\right.
	\end{aligned}
	\end{equation*}
and very recently in \cite{Youssfi_Mahmoud}, Youssfi and Mahmoud extended the result of \cite{Arcoya_merida} to the fractional Laplacian and established the following 
\[u \in W_0^{1,\frac{Nr(1+\gamma)}{N-rs(1-\gamma)}}(\Omega)  \  \text{ if } \ 0<\gamma <1 \ \text{and} \ f \in L^r(\Omega) \ \text{with} \ r \in [1,(2_s^\ast/(1-\gamma))'),\]
and for a non-negative datum $f \in L^r(\Omega)$, $r>1$ and $\gamma>1$
\[ u^\alpha \in H_0^s(\Omega) \ \text{for all} \ \alpha \in \left(\max\left(\frac{1}{2},\frac{sr(1+\gamma)-r+1}{2rs}\right), \frac{1+\gamma}{2}\right].\]
In case of $f \sim \delta^{-\zeta}$ for some $\zeta \in [0,2s)$, Adimurthi et al. \cite{Adi.al_2018} and Arora et al. \cite{Arora.al_2020} studied the same problem with singular nonlinearites in case of $N >2s$ and $N=2s$ respectively, and discussed the existence and uniqueness results of the classical solutions with respect to the singular parameters. Moreover, using the integral representation via Green function and maximum principle, they proved the sharp boundary behavior of the weak solution. For further issues on nonlocal and nonlinear singular problems, the interested reader can consult to the
bibliographic references in  \cite{Arora_et-al-1, Arora_2, Canino_et_al, Cam_et_al, Giacomoni.al_2017}. \vspace{0.1cm}\\
\noindent  \textbf{Notations:} Throughout the paper, we assume that $\Omega \subset \mathbb{R}^N$ ($N \geq 2$) is a bounded domain with $C^{1,1}$ boundary. Set $\delta(x) := \dist(x, \partial \Omega)$ and $\mathcal{D}_{\Omega}=\diam(\Omega).$ For $i \in \mathbb{N}$, we denote by $C_i, c_i, d_i$ positive constants that may vary from line to line. If necessary, we will write $C = C(a, b)$ to emphasize the dependence
of $C$ on $a, b.$ For a number $q \in (1, \infty)$, we denote by $q'$
the conjugate exponent of $q$, namely $q'=\frac{q}{q-1}$. For two functions $f, g$, we write $f \lesssim g$ or $f \gtrsim g$ if there exists a constant
$C > 0$ such that $f \leq Cg$ or $f \geq Cg$. We write $f \sim g $ if $f \lesssim g$ and $g \gtrsim f$. For $\eta >0$, $\chi_{\Omega_\eta}$ denotes the characteristic function of $\Omega_\eta:= \{x \in \Omega: \delta(x) < \eta\}.$ Denote
$$r^\sharp:= \left(\frac{2^*}{(1-\gamma)}\right)' \ \text{for} \ 0\leq \gamma< 1 \ \text{and} \ 2^* := \frac{2N}{N-2}$$
$$\mathcal{P}_{r, \gamma}:= \{(r, \gamma): \ \text{$1 \leq r \leq + \infty, \ \gamma \geq 0\ \text{and} \ (r, \gamma) \neq (1,0)$}\},$$
and 
\[L^r_c(\Omega):= \{f \in L^r(\Omega): \supp(f) \Subset \Omega\}.\]
For $\gamma >0$ and $\zeta \geq 0$, we define a class of functions
$$\mathcal{A}_\zeta(\Omega):=\left\{f:\Omega \to \mathbb{R}^+ \cup \{0\}:  f \asymp \delta^{-\zeta} \right\}$$
and for $\zeta \neq 2$, denote 
$$\mathfrak{L}^*:= \frac{\gamma + \zeta-1}{2-\zeta}.$$
\subsection{Description of main results}
In the present work, we derive the qualitative properties of the weak solution to a mixed type elliptic problem for two different class of weight functions $f$, and in both presence or absence of singular nonlinearities {\it i.e.} $\gamma >0$ or $\gamma=0$ respectively. In this part, we give a short description of our main results and for a detailed presentation, we refer the reader to Section \ref{main:res}. \vspace{0.1cm}\\
For the first class of weight function $f$, {\it i.e.} $f \in L^r(\Omega)$ for $1 \leq r \leq \infty$, we show:
\begin{itemize}
    \item \textbf{Existence results:} For this, we use classical approach of regularizing the singular nonlinearities $u^{-\gamma}$ by $(u+\frac{1}{n})^{-\gamma}$ and derive uniform a priori estimates for the weak solution of the regularized problem. The crucial step here is to choose an appropriate test function in the energy space. By taking into the account the combined interaction of the summability of the datum $f$ and the singular exponent $\gamma$, we obtain our existence results in two disjoint subsets of our admissible set $\mathcal{P}_{r, \gamma}$ with different Sobolev regularity: 
    \begin{equation*}
	u \in 
	 \left\{
	\begin{array}{ll}
	W_0^{1,q}(\Omega) & \text{ if } \ (r,\gamma) \in \mathcal{P}_{r, \gamma} \cap \{(r, \gamma): r \in [1, r^\sharp), 0 \leq \gamma <1\}, \vspace{0.1cm}\\
	H^1_{loc}(\Omega) & \text{ if } \ (r,\gamma) \in \mathcal{P}_{r, \gamma} \setminus \{(r, \gamma): r \in [1, r^\sharp), 0 \leq \gamma < 1\},
	\end{array} 
	\right.
	\end{equation*}
with $q:= \frac{Nr(1+\gamma)}{N-r(1-\gamma)}.$ Here, the notion of solution in two disjoints subsets may differ due to different Sobolev regularity of the solution. Moreover, we observe that there is a kind of continuity in the summability exponents in the sense that as $r \to (r^\sharp)^-$, $q \to 2$ when $0 \leq \gamma < 1.$ These existence results for mixed type operator also complements the results for classical Laplacian in \cite{Bocc_ors} and fractional Laplacian in \cite{Barrios.al_2015, Youssfi_Mahmoud}. 
    \item \textbf{Power and exponential type Sobolev regularity of the weak solution}: Here the term power type Sobolev regularity means that $u^\alpha \in H_0^1(\Omega)$ for some $\alpha >0$ and the 
   exponential type Sobolev regularity means that there exists $\beta >0$ such that $$\exp(\beta u)-1 \in H_0^1(\Omega) \ \text{when} \ 0 \leq \gamma \leq 1,$$ and for $\tau \geq \gamma$, $$\left(\exp(\beta u)-1\right)^\tau \in H_0^1(\Omega) \ \text{when} \ \gamma >1.$$
   For this, we prove two types of power type Sobolev regularity result, depending upon the value of $\alpha,$ one with the help of appropriate choice of test functions when $\alpha$ is large and the second by using the lower boundary behavior of the approximating solution when $\alpha$ is small. \vspace{0.1cm}\\
   \textit{Type 1:} When $1 \leq r < \frac{N}{2}$, $\gamma \geq 0$, we show that
   \begin{equation*}
	\text{$u^{\alpha} \in H_0^1(\Omega)$ \ for any \ $\alpha \in \left[\frac{\gamma+1}{2}, \frac{\mathfrak{S}_r+1}{2}\right]$ \ where \ $ \mathfrak{S}_r:=\frac{N(r-1) + \gamma r(N-2)}{N-2r}$}.
	\end{equation*}
	We notice that as $r \to \frac{N}{2}^-$, $\mathfrak{S}_r \to \infty$ and so it natural to expect the exponential type Sobolev regularity when $r=\frac{N}{2}.$ In this regard, for $r= \frac{N}{2}$, we prove the exponential type Sobolev regularity in the sense mentioned above.  The first step in establishing such regularity results is to find an appropriate test function in the energy space to handle the singular nonlinearities and second step is to derive uniform a priori estimates for the approximating sequence with the same power type Sobolev regularity and then pass to the limits. This type of Sobolev regularity result are even new for the classical Laplacian and fractional Laplacian singular and non-singular problems.  
	\vspace{0.1cm}\\
   \textit{Type 2:} In this case, to handle the singular nonlinearities, we exploit the boundary behavior of the weak solution in deriving uniform a priori estimates. Precisely, we show that for any $r>1$ and $\gamma>0$,
   \begin{equation*}
	u^\alpha \in H_0^1(\Omega) \ \text{for any} \ \alpha \in 
	 \left\{
	\begin{array}{ll}
	\left(\frac{1}{2},\frac{\gamma+1}{2}\right] & \text{ if } \ \gamma+\frac{1}{r} < 1, \vspace{0.1cm}\\
	 \left(\frac{r\gamma+ 1}{2r},\frac{\gamma+1}{2}\right] & \text{ if } \ \gamma+\frac{1}{r} \geq  1.
	\end{array} 
	\right.
	\end{equation*}
In addition to that, we also have shown that if $\alpha \leq \frac{1}{2}$, then $u^\alpha \not\in H_0^1(\Omega)$ if $\gamma+\frac{1}{r} < 1$, which highlight the optimality of this result.
\item \textbf{Continuity with respect to datum}: An an application of the \textit{Type 1} power-type Sobolev regularity results, we show that for any $1 \leq r < \frac{N}{2}$ and $\mathfrak{S} \in [\gamma, \mathfrak{S}_r]$ and for minimal weak solution $u$ and $v$ with respect to datum $f$ and $g$ respectively, the following inequality holds 
\[
    \|\nabla |u-v|^{\frac{\mathfrak{S}+1}{2}}\|^2_{L^2(\Omega)}  \leq C~ \|f-g\|_{L^r(\Omega)}^{\frac{r(N-2)}{N-2r}} .
\]
The above inequality also implies the continuity of the solution with respect to given datum and comparison estimates. 
\end{itemize}
For the second class of weight function $f \in \mathcal{A}_\zeta(\Omega)$ for $\zeta \geq 0$, we show
\begin{itemize}
    \item \textbf{Existence results:} For this, we have followed the same classical method of regularizing the singular problem, but here to prove the uniform a priori estimates of the approximating sequence and handle singular nonlinearities, we cannot use the same approach of exploiting Lebesgue summability of the datum $f$, since for $\zeta \geq 1$, $f \not\in L^r(\Omega)$ for any $r \geq 1$. To resolve this issue, we prove new boundary estimates of the approximating sequence, by using the lower and upper bound estimates of the Green's kernel associated to the mixed operator. By considering the interplay of both singular exponents $\gamma>0$ and $\zeta \in [0,2)$, we obtain the following existence results with different Sobolev regularity: 
      \begin{equation*}
	u \in 
	 \left\{
	\begin{array}{ll}
	H_0^{1}(\Omega) & \text{ if } \ \gamma+ \zeta \leq 1, \vspace{0.1cm}\\
	H^1_{loc}(\Omega) & \text{ if } \ \gamma+ \zeta > 1.
	\end{array} 
	\right.
	\end{equation*}
    \item \textbf{Optimal boundary behavior:} To prove this, we study the action of the Green's operator on the inverse of the distance function perturbed with logarithmic nonlinearity. Using the lower and upper bound estimates of the Green's kernel \cite{Chenkim_Green1} and borrowing some techniques from \cite{AbaGomVaz_2019}, we show
    \[ \left\{
	\begin{array}{ll}
		u \asymp \delta & \text{ if } \ \zeta+\gamma <1,\vspace{0.1cm}\\
		u \asymp \delta \ln^{\frac{1}{2-\zeta}}(\frac{\mathcal{D}_\Omega}{\delta}) & \text{ if } \ \zeta+\gamma=1 , \vspace{0.1cm}\\
	u \asymp \delta^{\frac{2-\zeta}{\gamma+1}}  & \text{ if } \ \zeta+\gamma >1.
	\end{array} 
	\right.\]
    \item \textbf{Optimal power type Sobolev regularity and non-existence results}: As an application of the above optimal boundary behavior, we prove
    \[
    u^{\frac{\mathfrak{L}+1}{2}} \ \text{belongs to} \  H_0^1(\Omega) \quad \text{if and only if} \quad \mathfrak{L} > \left\{
	\begin{array}{ll}
	0  & \text{ if } \  \zeta+ \gamma \leq  1, \vspace{0.1cm}\\
	\mathfrak{L}^* & \text{ if } \ \zeta+ \gamma > 1.
	\end{array} 
	\right.
    \]
    and for $\zeta \geq 2$ and $\gamma>0$, non-existence results is established. These results for mixed type operator in this class complements the results for  fractional Laplacian in \cite{Adi.al_2018, Arora.al_2020}. 
\end{itemize}
The plan of the article is organized as follows: In section \ref{main:res}, we provide the detailed statement of main results of this work. In section \ref{prelim:res}, we prove some preliminary results for the approximated problem which will be used in the rest of the work. Section \ref{Sec:apri} is devoted to deriving uniform a priori estimates, Green's function estimates and boundary behavior of the solution of the approximated problem. In section \ref{sec:proof}, we provide the proof of our main results. At the end, we give a short appendix section \ref{sec:append}.

\section{Statement of main results}\label{main:res}
In this section, we write a detailed statement of our main results.
\subsection{Lebesgue weights}
We start with presenting our existence results:
\begin{thm}\label{result:1}
Assume that $f \in L^r(\Omega) \setminus \{0\}$ is non-negative and $(r, \gamma) \in \mathcal{P}_{r,\gamma} \cap \{(r, \gamma): r \in [1, r^\sharp), 0 \leq \gamma <1\}$. Then there exists a positive weak solution $u$ of the problem \eqref{main:problem} in the following sense: 
\begin{enumerate}
    \item[(i)] $u \in W_0^{1,q}(\Omega)$ with $q:= \frac{Nr(1+\gamma)}{N-r(1-\gamma)}.$
    \item[(ii)] for every $\omega \Subset \Omega$, there exists a constant $C=C(\omega)$ such that $
0< C(\omega) \leq u.$
   \item[(iii)] for every $\psi \in W_0^{1,q'}(\Omega) \cap L^{r'}_c(\Omega)$
\begin{equation}\label{weak:formula1}
\begin{split}
    \int_{\Omega} \nabla u \cdot \nabla \psi ~dx + \frac{C(N,s)}{2} \int_{\mathbb{R}^N} \int_{\mathbb{R}^N} \frac{(u(x)- u(y))(\psi(x)- \psi(y))}{|x-y|^{N+2s}} ~dx ~dy = \int_{\Omega}  \frac{f(x)}{u^\gamma}\psi ~dx.
\end{split}
\end{equation}
\end{enumerate}
\end{thm}
\begin{thm}\label{result:2}
Assume that $f \in L^r(\Omega) \setminus \{0\}$ is non-negative and $(r, \gamma) \in \mathcal{P}_{r,\gamma} \setminus \{(r, \gamma): r \in [1, r^\sharp), 0 \leq \gamma <1\}$. Then there exists a positive weak solution $u$ of the problem \eqref{main:problem} in the following sense: 
\begin{enumerate}
    \item[(i)] $u \in H^1_{loc}(\Omega)$ and $u=0 \ \text{on} \ \Omega \ \text{in the sense of Definition \ref{def:boundary}}.$
    \item[(ii)] for every $\omega \Subset \Omega$, there exists a constant $C=C(\omega)$ such that $
0< C(\omega) \leq u.$
\item[(iii)] for every $\psi \in \bigcup_{\tilde{\Omega} \Subset \Omega} H^1_{loc}(\tilde{\Omega}) \cap L^{r'}_c(\Omega)$
\begin{equation}\label{weak:formula2}
\begin{split}
    \int_{\Omega} \nabla u \cdot \nabla \psi ~dx + \frac{C(N,s)}{2} \int_{\mathbb{R}^N} \int_{\mathbb{R}^N} \frac{(u(x)- u(y))(\psi(x)- \psi(y))}{|x-y|^{N+2s}} ~dx ~dy = \int_{\Omega}  \frac{f(x)}{u^\gamma}\psi ~dx.
\end{split}
\end{equation}
\end{enumerate}
\end{thm}
Now, we state the regularity results showing power and exponential type Sobolev regularity depending upon the summability of the datum $f$ and the singular exponent $\gamma \geq 0$.
\begin{thm}\label{result:3}
Assume that $f \in L^r(\Omega) \setminus \{0\}$ is non-negative and $(r, \gamma) \in \mathcal{P}_{r, \gamma}$. Let $u$ be a weak solution of the problem \eqref{main:problem} obtained in Theorem \ref{result:1} and Theorem \ref{result:2}. Then,
\begin{itemize}
    \item[(i)]\textbf{(Weak case)} For $r \in [1, \frac{N}{2})$ and $N > 2$, 
    \begin{equation*}
	\text{$u^{\frac{\mathfrak{S}+1}{2}} \in H_0^1(\Omega)$ \ for any \ $\mathfrak{S} \in \left[\gamma,\mathfrak{S}_r\right]$ \ where \ $ \mathfrak{S}_r:=\frac{N(r-1) + \gamma r(N-2)}{N-2r}$}.
	\end{equation*}
	Moreover, $u$ belongs to $L^{\sigma_r}(\Omega)$ with $\sigma_r:= \frac{Nr(1+\gamma)}{N-2r}.$
    \item[(ii)]\textbf{(Limit case)} For $r= \frac{N}{2}$ and $N \geq 2$. Then 
    \begin{enumerate}
    \item \textbf{(Weak singularity)}  For $0 \leq \gamma \leq 1$,
    $$\mathfrak{H} \left(\frac{u}{2}\right) \in  H_0^1(\Omega) \quad \text{where} \quad \mathfrak{H}(t):= \exp\left(\beta t\right)-1$$
    where $\beta >0 $ such that \begin{equation}\label{range:est}
	\frac{2}{S(N) \|f\|_{\frac{N}{2}}} \geq
	 \left\{
	\begin{array}{ll}
	\beta \max\{1, \left(\beta\gamma^{-1}\right)^\gamma\} & \text{ if } \ 0 < \gamma \leq 1, \vspace{0.1cm}\\
	\beta & \text{ if } \ \gamma=0.
	\end{array} 
	\right.
	\end{equation}
Moreover, there exist constants  $C_1, C_2$ depending upon $\beta, S(N),f, |\Omega|$ such that  $$\int_{\Omega} \exp\left(\frac{ \beta N u}{N-2}\right) \leq C_1 \quad \text{when} \ N >2 \quad \text{and} \quad \|u\|_{L^\infty(\Omega)} \leq C_2 \quad  \text{when} \ N =2.$$
        \item \textbf{(Strong singularity)}  For $\tau \geq \gamma > 1$ and $\alpha >0 $ such that $\frac{\alpha}{2^{3-2\tau}} \max\{1, \alpha^\tau\}< \frac{1}{\tau S(N) \|f\|_{\frac{N}{2}}}$ 
    $$\mathfrak{D} \left(\frac{u}{2}\right) \in  H_0^1(\Omega) \quad \text{where} \quad \mathfrak{D}(t) = \left(\exp(\alpha t) -1\right)^{\tau}$$
and there exist constants  $C_3, C_4$ depending upon $\alpha, \tau, S(N),f, |\Omega|$ such that  $$\int_{\Omega} \exp\left(\frac{ 2N \alpha \tau u}{N-2}\right) \leq C_3 \quad  \text{when} \ N >2 \quad \text{and} \quad \|u_n\|_{L^\infty(\Omega)} \leq C_4 \quad  \text{when} \ N =2.$$
    \end{enumerate}
In particular, $u$ belongs to $L^s(\Omega)$ for every $s \in [1, \infty)$ and $N >2.$
    \item[(iii)]\textbf{(Strong case)} For $r > \frac{N}{2}$ and $N \geq 2$, $u$ belongs to $L^\infty(\Omega).$
    \item[(iv)]\textbf{(Exact Sobolev regularity)} For $r \in [1, r^\sharp)$, $0 \leq \gamma < 1$ and $N > 2$, $u$ belongs to $W_0^{1, q}(\Omega)$ with $q:= \frac{Nr(1+\gamma)}{N-r(1-\gamma)}$ and for any $r \in [1, \infty]$, $\gamma=1$ or $r \geq r^\sharp$ and $0 \leq \gamma < 1$, $u$ belongs to $H_0^1(\Omega).$  
    \end{itemize}
\end{thm}
\begin{thm}\label{result:10}
Assume that $f \in L^r(\Omega) \setminus \{0\}$ is non-negative. Let $u$ be a weak solution of the problem \eqref{main:problem} obtained in Theorem \ref{result:1} and Theorem \ref{result:2}. Then,
    \begin{equation*}
	\text{$u^{\frac{\mathfrak{S}+1}{2}}$ belongs to \  $H_0^1(\Omega)$ \ for any \ $\mathfrak{S} \in \left(0,\gamma\right]$ \ if and only if \ $\gamma+\frac{1}{r} < 1$}.
	\end{equation*}
and 
    \begin{equation*}
	\text{$u^{\frac{\mathfrak{S}+1}{2}}$ belongs to \  $H_0^1(\Omega)$ \ for any \ $\mathfrak{S} \in \left(\gamma-1+\frac{1}{r},\gamma\right]$ \ if \ $\gamma+\frac{1}{r} \geq  1$}.
	\end{equation*}
\end{thm}
\begin{thm}\label{result:8}
Let $u$ and $v$ be two solutions of the problem \eqref{main:problem}-\eqref{bdry:cond} obtained in Theorem \ref{result:1} and Theorem \ref{result:2}  with respect to datum $f$ and $g$ in $L^r(\Omega)$ for $1 \leq r< \frac{N}{2}$ respectively. Then, for any $\mathfrak{S} \in [\gamma, \mathfrak{S}_r]$, there exists a constant $C=C(S(N), \mathfrak{S})$ independent of $u,v$ such that
\[
\begin{split}
    \|\nabla |u-v|^{\frac{\mathfrak{S}+1}{2}}\|^2_{L^2(\Omega)} & \leq C~ \|f-g\|_{L^r(\Omega)}^{\frac{r(N-2)}{N-2r}} .
\end{split}
\]
In addition to above, we have
\[ \int_{\Omega} |\nabla (u-v)_+^{\frac{\mathfrak{S}_r+1}{2}}|^2 ~dx \leq \int_{\Omega} (f(x)-g(x)) (u-v)_+^{\mathfrak{S}_r-\gamma} ~dx.\]
\end{thm}
\begin{cor}\label{result:9}
Under the condition of Theorem \ref{result:8}, if $f \leq g$, then $u \leq v$ a.e. in $\Omega.$
\end{cor}
\begin{remark}
Adopting the same arguments of \cite[Theorem 1.2]{Can_sci} for the Laplacian and \cite[Theorem 1.4]{Canino_et_al} for fractional Laplacian, we can obtain the uniqueness of weak solution proved in Theorem \ref{result:2}. The crucial step is to exploit the $H^1_{loc}(\Omega)$ regularity of weak solution $u$. To avoid repeating the same steps, we skip the proof.  
\end{remark}
\begin{remark}
Let $(r, \gamma) \in \mathcal{P}_{r, \gamma}$ such that $\mathfrak{S}_r \geq 1.$ In this case, the weak solution $u \in H_0^1(\Omega)$ therefore by classical density arguments, the class of test function in the weak formulation can be extended to $H_0^1(\Omega).$
\end{remark}
\subsection{Singular weights}
In this case of weight function, we work with $N \geq 3$. First, we start with presenting our existence results:
\begin{thm}\label{result:4}
Let $\gamma >0$ and $ \zeta \in [0,2)$ such that $\gamma+ \zeta \leq 1$, and $f \in \mathcal{A}_\zeta(\Omega)$. Then there exists a positive minimal solution $u$ of the problem \eqref{main:problem} in the following sense: 
\begin{enumerate}
    \item[(i)] $u \in H_0^{1}(\Omega).$
  \item[(ii)] for every $\psi \in H_0^{1}(\Omega)$ 
\begin{equation}\label{weak:formula3}
\begin{split}
    \int_{\Omega} \nabla u \cdot \nabla \psi ~dx + \frac{C(N,s)}{2} \int_{\mathbb{R}^N} \int_{\mathbb{R}^N} \frac{(u(x)- u(y))(\psi(x)- \psi(y))}{|x-y|^{N+2s}} ~dx ~dy = \int_{\Omega}  \frac{f(x)}{u^\gamma}\psi ~dx.
\end{split}
\end{equation}
\end{enumerate}
\end{thm}
\begin{thm}\label{result:5}
Let $\gamma >0$ and $ \zeta \in [0,2)$ such that $\gamma+ \zeta > 1$ and $f \in \mathcal{A}_\zeta(\Omega)$. Then there exists a positive minimal solution $u$ of the problem \eqref{main:problem} in the following sense: 
\begin{enumerate}
    \item[(i)] $u \in H^1_{loc}(\Omega)$ and $u=0 \ \text{on} \ \Omega \ \text{in the sense of Definition \ref{def:boundary}}.$
    \item[(ii)] for every $\omega \Subset \Omega$, there exists a constant $C=C(\omega)$ such that $
0< C(\omega) \leq u.$
\item[(iii)] for every $\psi \in  H^1_0(\Omega)$ in case of $\mathfrak{L}^* \leq 1$ and $\psi \in \bigcup_{\tilde{\Omega} \Subset \Omega} H^1_{loc}(\tilde{\Omega})$ with $\supp(\psi) \Subset \Omega$ in case of $\mathfrak{L}^* > 1$,
\begin{equation}\label{weak:formula4}
\begin{split}
    \int_{\Omega} \nabla u \cdot \nabla \psi ~dx + \frac{C(N,s)}{2} \int_{\mathbb{R}^N} \int_{\mathbb{R}^N} \frac{(u(x)- u(y))(\psi(x)- \psi(y))}{|x-y|^{N+2s}} ~dx ~dy = \int_{\Omega}  \frac{f(x)}{u^\gamma}\psi ~dx.
\end{split}
\end{equation}
\end{enumerate}
\end{thm}
Now, we state the regularity results displaying optimal power type Sobolev regularity results and optimal boundary behavior of minimal solution depending upon the singular exponents $\zeta$ and $\gamma$.
\begin{thm}\label{result:6}
Let $\gamma>0$, $\zeta \in [0,2)$ and $u$ be a weak solution of the problem \eqref{main:problem} obtained in Theorem \ref{result:4} and Theorem \ref{result:5}. Then,
\begin{equation*}
	u \ \text{belongs to} \  \mathcal{B}_{\gamma, \zeta}(\Omega) \ \text{where} \ 	\begin{aligned}
	\mathcal{B}_{\gamma, \zeta}(\Omega)=  \left\{
	\begin{array}{ll}
		u: u \asymp \delta & \text{ if } \ \zeta+\gamma <1,\vspace{0.1cm}\\
		u: u \asymp \delta \ln^{\frac{1}{2-\zeta}}(\frac{\mathcal{D}_\Omega}{\delta}) & \text{ if } \ \zeta+\gamma=1 , \vspace{0.1cm}\\
	u: u \asymp \delta^{\frac{2-\zeta}{\gamma+1}}  & \text{ if } \ \zeta+\gamma >1,
	\end{array} 
	\right.
	\end{aligned}
	\end{equation*}
	and     \begin{equation*}
	u^{\frac{\mathfrak{L}+1}{2}} \ \text{belongs to} \  H_0^1(\Omega) \quad \text{if and only if} \quad \mathfrak{L} > \left\{
	\begin{array}{ll}
	0  & \text{ if } \  \zeta+ \gamma \leq  1, \vspace{0.1cm}\\
	\mathfrak{L}^* & \text{ if } \ \zeta+ \gamma > 1.
	\end{array} 
	\right.
\end{equation*}
\end{thm}
As an applications of above theorem, we prove the following non-existence result:
\begin{thm}\label{result:7}
Let $\zeta \geq 2.$ Then, there does not exists a weak solution of the problem \eqref{main:problem} in the sense of Theorem \ref{result:4} and Theorem \ref{result:5}.
\end{thm}
\begin{remark}
Repeating the same proof \cite[Theorem 1.1]{Arora_et-al-1}, we can obtain the uniqueness of weak solution proved in Theorem \ref{result:4} and Theorem \ref{result:5} when $\gamma >0$ and $\beta \in [0, \frac{3}{2})$.
\end{remark}
\begin{remark}
Let $\zeta \in [0,2), \gamma >0$ such that $\gamma+\zeta \leq 1$, and $\zeta+\gamma>1$ and $\mathfrak{L}^* < 1.$ In this case, the weak solution $u \in H_0^1(\Omega)$ therefore by classical density arguments, the class of test function in the weak formulation can be extended to $H_0^1(\Omega).$ The condition $\mathfrak{L}^* < 1$ can be rewritten as $\gamma + 2 \zeta <3$ which further reduces to the classical Lazer-Mckeena obstruction in case of $\zeta=0.$
\end{remark}
\section{Preliminary lemmas}\label{prelim:res}
\begin{Lem}\label{lem:prelimi}
Let $h \in L^\infty(\Omega)$, $h \geq 0$ and $h \not \equiv 0$. Then the problem 
\begin{equation}\tag{$S$} \left\{
\begin{aligned}
-\Delta u + (-\Delta)^s u& =h, \quad u>0 \quad &&\text{ in }  \Omega, \\
u & = 0 &&  \text{ in } \ \mathbb{R}^N \setminus \Omega,
\end{aligned}
\right.
\end{equation}
admits a unique positive weak solution $u$. Moreover, $u \in L^\infty(\Omega) \cap C^{1, \beta}(\overline{\Omega})$ for every $\beta \in (0,1).$ 
\end{Lem}
\begin{proof}
To prove the existence result, we use the classical minimization arguments from the Calculus of variations. For the sake of completeness, we provide a brief sketch of the proof. For any $h \in L^\infty(\Omega)$, $h \geq 0$ and $h \not\equiv 0$, we define the energy functional $\mathcal{I}_h: H_0^1(\Omega) \to \mathbb{R}$ such that
$$\mathcal{I}_h(u):= \frac{1}{2} \int_{\Omega} |\nabla u|^2 ~dx + \frac{C(N,s)}{4} \int_{\mathbb{R}^N} \int_{\mathbb{R}^N} \frac{|u(x)-u(y)|^2}{|x-y|^{N+2s}} ~dx ~dy - \int_{\Omega} h u  ~dx.$$
First, we find the minimizer of the above energy functional and then look for the solution of the problem $(S)$ as a critical point of $\mathcal{I}_h.$ In fact, using \cite[Proposition 2.2]{Nezza_palatucci_valdi} and Sobolev's embedding, we observe that $\mathcal{I}_h$ is well-defined and
$$\mathcal{I}_h(u) \geq \frac{1}{2} \|\nabla u\|_2^2 -  |\Omega|^\frac{1}{2} \|h\|_{\infty} \|u\|_{2} \geq \|\nabla u\|_2 \left(\frac{1}{2} \|\nabla u\|_2 - S(N)  |\Omega|^\frac{1}{2} \|h\|_{\infty}\right) \to \infty \ \text{as} \ \|\nabla u\|_2 \to \infty$$
where $S(N)$ is the Sobolev constant. This implies the energy functional $\mathcal{I}_h$ is coercive. Moreover, $\mathcal{I}_h$ is a $C^1$  and convex energy functional. Thus, $\mathcal{I}_h$ is weakly lower semi-continuous. Combining all the above properties of $\mathcal{I}_h$, there exists a minimizer $u \in H_0^1(\Omega)$ and which is also a critical point of $\mathcal{I}_h$ {\it i.e.} 
\begin{equation}\label{def-approx}
\begin{split}
    \int_{\Omega} \nabla u \cdot \nabla \psi ~dx + \frac{C(N,s)}{2} \int_{\mathbb{R}^N} \int_{\mathbb{R}^N} \frac{(u(x)- u(y))(\psi(x)- \psi(y))}{|x-y|^{N+2s}} ~dx ~dy = \int_{\Omega} h \psi ~dx \quad \forall \psi \in H_0^1(\Omega).
\end{split}
\end{equation}
By taking $\psi= u^-:= \min\{u, 0\} \in H_0^1(\Omega)$ as a test function in \eqref{def-approx} and using the fact that $h \geq 0$ and $$(u(x)-u(y))(u^-(x)- u^-(y)) \geq 0 \quad \text{for all} \ (x,y) \in \mathbb{R}^N \times \mathbb{R}^N,$$
we obtain
$$\int_{\Omega} |\nabla u^-|^2 ~dx \leq 0.$$
Hence, $u \geq 0$ a.e. in $\Omega.$ Now, we show that the problem $(S)$ has a unique solution. Let $u_1, u_2 \in H_0^1(\Omega)$ be two solution of the problem $(S)$. Therefore, for all $\psi \in H_0^1(\Omega)$, we have
\begin{equation}\label{def-approx-1}
\begin{split}
    \int_{\Omega} \nabla u_1 \cdot \nabla \psi ~dx + \frac{C(N,s)}{2} \int_{\mathbb{R}^N} \int_{\mathbb{R}^N} \frac{(u_1(x)- u_1(y))(\psi(x)- \psi(y))}{|x-y|^{N+2s}} ~dx ~dy = \int_{\Omega} h \psi ~dx,
\end{split}
\end{equation}
\begin{equation}\label{def-approx-2}
\begin{split}
    \int_{\Omega} \nabla u_2 \cdot \nabla \psi ~dx + \frac{C(N,s)}{2} \int_{\mathbb{R}^N} \int_{\mathbb{R}^N} \frac{(u_2(x)- u_2(y))(\psi(x)- \psi(y))}{|x-y|^{N+2s}} ~dx ~dy = \int_{\Omega} h \psi ~dx.
\end{split}
\end{equation}
By subtracting the above two equations and inserting $\psi= u_1-u_2$, we obtain
\begin{equation*}
\begin{split}
    \int_{\Omega} |\nabla (u_1- u_2)|^2  ~dx + \frac{C(N,s)}{2} \int_{\mathbb{R}^N} \int_{\mathbb{R}^N} \frac{|(u_1- u_2)(x)- (u_1- u_2)(y)|^2}{|x-y|^{N+2s}} ~dx ~dy =0.
\end{split}
\end{equation*}
which further gives $u_1=u_2$ a.e. in $\Omega.$
The boundedness and regularity of solution follows by employing the classical method of Stampacchia see for e.g \cite[Theorem 4.7]{Biagi_Valdi} and using \cite[Theorem 2.7]{Biagi_Valdi-1}. Since $h \not\equiv 0$, we have $u \not \equiv 0$. Finally, the strong maximum principle in \cite[Theorem 3.1]{Biagi_Vecchi}, implies $u> 0$ in $\Omega.$ 
\end{proof}
Depending upon the class of weight function $f$, we consider a sequence of increasing function $f_n$ such that $f_n \to f$ $a.e.$ in  $\Omega$. In the first case, when $f \in L^r(\Omega)$, we consider $f_n(x):=\min\{f(x), n\}$ and in second case, when $f \in \mathcal{A}_\zeta$, we consider
	\begin{equation*}
	f_n(x):=
	 \left\{
	\begin{array}{ll}
	\left(f^{\frac{-1}{\zeta}}(x) + \left(\frac{1}{n}\right)^\frac{\gamma+1}{2-\zeta}\right)^{-\zeta} & \text{ if } \ x \in \Omega , \vspace{0.1cm}\\
	0 & \text{else},
	\end{array} 
	\right.
	\end{equation*}
and there exist positive constants $\mathcal{G}_1, \mathcal{G}_2>0$ 	such that, for any $x \in \Omega$
\begin{equation}\label{weight:approx}
\frac{\mathcal{G}_1}{\left(d(x) + \left(\frac{1}{n}\right)^\frac{\gamma+1}{2-\zeta}\right)^{\zeta}}  \leq f_n(x) \leq   \frac{\mathcal{G}_2}{\left(d(x) + \left(\frac{1}{n}\right)^\frac{\gamma+1}{2-\zeta}\right)^{\zeta}}.
\end{equation}
By considering the above choice of $f_n$, we study the following approximated singular problem
\begin{equation}\label{sing:approx} \tag{$P_n$} \left\{
\begin{aligned}
-\Delta u + (-\Delta)^s u& =\frac{f_n}{(u+\frac{1}{n})^\gamma}, \quad u>0 \quad &&\text{ in }  \Omega, \\
u & = 0 &&  \text{ in } \ \mathbb{R}^N \setminus \Omega .
\end{aligned}
\right.
\end{equation}
where $\{f_n\}_{n \in \mathbb{N}}$ is a bounded increasing sequence such that $f_n \to f$ in $L^r(\Omega)$ for $r \in [1, \infty)$ and $f_n=f$ when $r=+\infty.$
\begin{Lem}\label{Lem:apriori}
For any $n \in \mathbb{N}$ and $\gamma \geq 0$, there exists a unique non-negative weak solution $u_n \in H_0^1(\Omega)$ of the problem $(P_n)$ in the sense that
\begin{equation}\label{def:notion:approx}
\begin{split}
    \int_{\Omega} \nabla u_n \cdot \nabla \psi ~dx + \frac{C(N,s)}{2} \int_{\mathbb{R}^N} \int_{\mathbb{R}^N} &\frac{(u_n(x)- u_n(y))(\psi(x)- \psi(y))}{|x-y|^{N+2s}} ~dx ~dy \\
    &= \int_{\Omega} \frac{f_n(x)}{(u_n + \frac{1}{n})^\gamma} \psi ~dx \quad \forall \psi \in H_0^1(\Omega).
\end{split}
\end{equation}
Moreover,
\begin{itemize}
    \item[(i)] The solution $u_n \in C^{1, \beta}(\overline{\Omega}) \cap L^\infty(\Omega)$ for every $\beta \in (0,1)$ and $u_n >0$ in $\Omega.$ Moreover, when $f_n \in C^{\alpha}(\overline{\Omega})$, then $u_n \in C^{2, \delta}(\overline{\Omega}) \cap C(\overline{\Omega})$ for some $\delta \in (0,1).$
    \item[(ii)] The sequence $\{u_n\}_{n \in \mathbb{N}}$ is monotonically increasing in the sense that $u_{n+1} \geq u_n$ for all $n \in \mathbb{N}.$ 
    \item[(iii)] For every compact set $K \Subset \Omega$ and $n \in \mathbb{N}$, there exists a constant $C$ depending upon $K$ and independent of $n$ such that $u_n \geq C >0.$
\end{itemize}
\end{Lem}
\begin{proof}
Given $g \in L^2(\Omega)$ and $n \in \mathbb{N}$ fixed, set $$h:= \frac{f_n}{\left(g^+ + \frac{1}{n}\right)^\gamma}.$$ Then, in view of Lemma \ref{lem:prelimi}, there exists a unique positive bounded solution $w \in H_0^1(\Omega)$ for the problem $(S)$ (see statement of Lemma \ref{lem:prelimi}) with $h$ defined above. Therefore, we define a operator $T: L^{2}(\Omega) \to L^2(\Omega)$ such that $T(g)=w$ where $w$ satisfies
\begin{equation}\label{def:notion:approx-1}
\begin{split}
    \int_{\Omega} \nabla w \cdot \nabla \psi ~dx + \frac{C(N,s)}{2} \int_{\mathbb{R}^N} \int_{\mathbb{R}^N} \frac{(w(x)- w(y))(\psi(x)- \psi(y))}{|x-y|^{N+2s}} ~dx ~dy = \int_{\Omega} \frac{f_n(x)}{(g^+ + \frac{1}{n})^\gamma} \psi ~dx \quad \forall \psi \in H_0^1(\Omega).
\end{split}
\end{equation}
Using $w$ as a test function in \eqref{def:notion:approx-1} and using Sobolev's embedding, we obtain
 \begin{equation}\label{est:apri}
 \int_{\Omega} |\nabla w|^2 ~dx \leq  \int_{\Omega} \frac{f_n(x) w}{(g^+ + \frac{1}{n})^\gamma} ~dx \leq R \|\nabla w\|_{2} \Longrightarrow 
  \|\nabla w\|_{2} \leq R
 \end{equation}
where $R:= S(N) \|f_n\|_{L^\infty(\Omega)} n^\gamma |\Omega|^\frac{1}{2}$. This implies that the ball $B(0,R)$ of radius $R$ in $H_0^1(\Omega)$ is invariant under the action of the map $T$. Now, we prove the continuity and compactness of the map $T: H_0^1(\Omega) \to H_0^1(\Omega)$ in order to apply the Schauder's fixed point theorem. For continuity, we claim that $w_k \to w$ in $H_0^1(\Omega)$ when $h_k \to h$ in $H_0^1(\Omega)$ with $w_k= T(h_k)$ and $w= T(h).$ Considering the corresponding the sequence $\{w\}_{k \in \mathbb{N}}$ and choosing $w_k-w$ as a test function, we get
\begin{equation}\label{est:passage}
\begin{split}
    \int_{\Omega} |\nabla w_k-w|^2 ~dx &= \int_{\Omega} \left(\frac{f_n(x)}{(h_k^+ + \frac{1}{n})^\gamma} - \frac{f_n(x)}{(h^+ + \frac{1}{n})^\gamma}\right) (w_k-w) ~dx\\
    & \leq S(N) \left(\int_{\Omega} \left(\frac{f_n(x)}{(h_k^+ + \frac{1}{n})^\gamma} - \frac{f_n(x)}{(h^+ + \frac{1}{n})^\gamma}\right)^{\frac{2N}{N+2}} ~dx\right)^\frac{N+2}{2N}\|\nabla w_k -\nabla w\|_{2}.
\end{split}
\end{equation}
Notice that, the integrand in the first term is dominated by $2 \|f_n\|_{\infty} n^\gamma$ and converge to $0$ a.e. in $\Omega$, then by applying dominated convergence theorem we obtain our claim. For compactness, let $h_k$ be a bounded sequence in $H_0^1(\Omega)$ and for $w_k= T(h_k)$, we claim that  $w_k \to w$ in $H_0^1(\Omega)$ up to a subsequence for some $w \in H_0^1(\Omega).$ In light of \eqref{est:apri}, both $w_k$ and $h_k$ are bounded in $H_0^1(\Omega)$, then up to a subsequence we have $$w_k \rightharpoonup w \ \text{in} \ H_0^1(\Omega), \ w_k \to w \ \text{ in} \ L^p(\Omega) \ \text{ for any} \ 1 \leq p < \frac{2N}{N-2} \ \text{and} \ h_k \to h \ \text{a.e. in} \ \Omega.$$
We also know that, $w_k$ satisfies: for any $\psi \in H_0^1(\Omega)$
\begin{equation}\label{est:apop}
\begin{split}
    \int_{\Omega} \nabla w_k \cdot \nabla \psi ~dx + \frac{C(N,s)}{2} \int_{\mathbb{R}^N} \int_{\mathbb{R}^N} \frac{(w_k(x)- w_k(y))(\psi(x)- \psi(y))}{|x-y|^{N+2s}} ~dx ~dy = \int_{\Omega} \frac{f_n(x)}{(h_k^+ + \frac{1}{n})^\gamma} \psi ~dx.
\end{split}
\end{equation}
Now to pass limits $k \to \infty$, we observe that
$$\frac{(w_k(x)- w_k(y))(\psi(x)- \psi(y))}{|x-y|^{\frac{N+2s}{2}}} \ \text{in uniformly bounded in} \ L^2(\mathbb{R}^N \times \mathbb{R}^N)$$
because of Sobolev's embedding and by the pointwise convergence of $w_k$ to $w$, we have
$$\frac{(w_k(x)- w_k(y))(\psi(x)- \psi(y))}{|x-y|^{\frac{N+2s}{2}}} \to  \frac{(w(x)- w(y))(\psi(x)- \psi(y))}{|x-y|^{\frac{N+2s}{2}}} \ \text{a.e. in} \ \Omega.$$
Then, since 
$$\frac{(\psi(x)- \psi(y))(\psi(x)- \psi(y))}{|x-y|^{\frac{N+2s}{2}}} \in L^2 (\mathbb{R}^N \times \mathbb{R}^N),$$ it is easy to see the passage of limit in \eqref{est:apop} to
\[
\int_{\Omega} \nabla w \cdot \nabla \psi ~dx + \frac{C(N,s)}{2} \int_{\mathbb{R}^N} \int_{\mathbb{R}^N} \frac{(w(x)- w(y))(\psi(x)- \psi(y))}{|x-y|^{N+2s}} ~dx ~dy = \int_{\Omega} \frac{f_n(x)}{(h^+ + \frac{1}{n})^\gamma} \psi ~dx.
\]
Now, by arguing as in \eqref{est:passage}, we get $w_k \to w$ in $H_0^1(\Omega)$ and Schauder's fixed point theorem implies the existence of a fixed points $u_n$ such that $T(u_n)= u_n$ for all $n \in \mathbb{N}$ {\it i.e.}
\[
\int_{\Omega} \nabla u_n \cdot \nabla \psi ~dx + \frac{C(N,s)}{2} \int_{\mathbb{R}^N} \int_{\mathbb{R}^N} \frac{(u_n(x)- u_n(y))(\psi(x)- \psi(y))}{|x-y|^{N+2s}} ~dx ~dy = \int_{\Omega} \frac{f_n(x)}{(u_n^+ + \frac{1}{n})^\gamma} \psi ~dx.
\]

The boundedness and regularity of solution follows by employing the classical method of Stampacchia see for e.g \cite[Theorem 4.7]{Biagi_Valdi} and using \cite[Theorem 2.7]{Biagi_Valdi-1} {\it i.e.} $u_n \in L^\infty(\Omega) \cap C^{1, \beta}(\overline{\Omega})$. Since $\frac{f_n}{\left(u^+_n + \frac{1}{n} \right)^\gamma} \not\equiv 0$, we have $u_n \not \equiv 0$. Finally, the strong maximum principle in \cite[Theorem 3.1]{Biagi_Vecchi}, implies $u_n> 0$ in $\Omega$ and $u_1 \geq C(K) >0$ for every $K$ compact subset of $\Omega.$ To prove the monotonicity, let $u_n$ and $u_{n+1}$ are positive solutions of the problem $(P_n)$ and $(P_{n+1})$ respectively {\it i.e.} for any $\psi \in H_0^1(\Omega)$, we have 
\[
\int_{\Omega} \nabla u_n \cdot \nabla \psi ~dx + \frac{C(N,s)}{2} \int_{\mathbb{R}^N} \int_{\mathbb{R}^N} \frac{(u_n(x)- u_n(y))(\psi(x)- \psi(y))}{|x-y|^{N+2s}} ~dx ~dy = \int_{\Omega} \frac{f_n(x)}{(u_n + \frac{1}{n})^\gamma} \psi ~dx
\]
and
\[
\int_{\Omega} \nabla u_{n+1} \cdot \nabla \psi ~dx + \frac{C(N,s)}{2} \int_{\mathbb{R}^N} \int_{\mathbb{R}^N} \frac{(u_{n+1}(x)- u_{n+1}(y))(\psi(x)- \psi(y))}{|x-y|^{N+2s}} ~dx ~dy = \int_{\Omega} \frac{f_{n+1}(x)}{(u_{n+1} + \frac{1}{n+1})^\gamma} \psi ~dx.
\]
Subtracting the above equalities by taking the test function $\psi= (u_n - u_{n+1})^+$ and using the following inequality (\cite[Lemma 9]{Lind_lind}): $\ \text{for a.e.} \ (x,y) \in \mathbb{R}^N \times \mathbb{R}^N$
$$\left(((u_n-u_{n+1})(x)- (u_n-u_{n+1})(y)) ((u_n-u_{n+1})^+(x)- (u_n-u_{n+1})^+(y))\right) \geq 0,$$
we get
\begin{equation*}
    \begin{split}
        \int_{u_n \geq u_{n+1}} |\nabla (u_n - u_{n+1})|^2 ~dx &\leq \int_{\Omega} \left( \frac{f_{n}(x)}{(u_{n} + \frac{1}{n})^\gamma} -\frac{f_{n+1}(x)}{\left(u_{n+1} + \frac{1}{n+1}\right)^\gamma}\right) (u_n - u_{n+1})^+ ~dx\\
        & \leq \int_{\Omega} f_{n+1}(x)  \frac{(u_{n} + \frac{1}{n})^\gamma- (u_{n+1} + \frac{1}{n+1})^\gamma}{\left((u_{n} + \frac{1}{n})(u_{n+1} + \frac{1}{n+1})\right)^\gamma} (u_n - u_{n+1})^+ ~dx \leq 0
    \end{split}  
\end{equation*}
which in turn implies that $u_n \leq u_{n+1}$ in $\Omega$ and $u_n \geq C(K)$ for every $n \in \mathbb{N}$ and $K \Subset \Omega.$
\end{proof}
Now, we derive the lower boundary behavior of the approximated sequence $u_n$ with the help of integral representation of the solution via Green's operator.
\begin{thm}\label{Green:thm:1}
Let $u_n$ be the weak solution of the problem $(P_n)$ obtained in Lemma \ref{Lem:apriori}, then there exist a constant $C_0>0$ independent of $n$ and $x\in \Omega$ such that $C_0 \delta(x) \leq u_n(x)$.
\end{thm}
\begin{proof}
Let $u_n$ be the sequence of weak solution of the problem $(P_n)$ and $G(x,y)$ the Green function associated to the mixed operator with homogeneous Dirichlet boundary conditions in $\Omega.$ Then, using the integral representation of the solution, we have
\begin{equation}\label{int:rep}
 u_n(x) := \mathbb{G}^\Omega\left[\frac{f_n(y)}{(u_n + \frac{1}{n})^\gamma}\right](x)= \int_{\Omega} \frac{G(x,y)  f_n(y)}{(u_n + \frac{1}{n})^\gamma} ~dy \quad \text{for} \ n \in \mathbb{N}.   
\end{equation}
From Lemma \ref{Lem:apriori}, we know that $\{u_n\}_{n \in \mathbb{N}}$ is an increasing sequence such that $u_n \in L^\infty(\Omega)$. Using this fact for any $n \in \mathbb{N}$ and $x \in \Omega$, we deduce that
\begin{equation}\label{est:appr:solu:small:lower}
	\begin{split}
	\delta(x) \lesssim \varphi_1(x) &= \lambda_1 \mathbb{G}^\Omega[\varphi_1](x) = \lambda_1 \mathbb{G}^\Omega\left[\varphi_1 \frac{(\|u_1\|_{L^\infty(\Omega)}+1)^\gamma}{(\|u_1\|_{L^\infty(\Omega)}+1)^\gamma}\right](x)\\
	& \leq \lambda_1 \|\varphi_1\|_{L^\infty(\Omega)} (\|u_1\|_{L^\infty(\Omega)}+1)^\gamma (\mathcal{D}_\Omega +1)^\zeta \mathbb{G}^\Omega\left[ \frac{f_1(x)}{(\|u_1\|_{L^\infty(\Omega)}+1)^\gamma}\right](x)\\
	& \leq C \mathbb{G}^\Omega\left[ \frac{f_1(x)}{(u_1+1)^\gamma}\right](x) \lesssim u_1(x) \leq u_n(x).
	\end{split} 
	\end{equation}
\end{proof}
 \section{Uniform apriori estimates}\label{Sec:apri}
\subsection{Lebesgue weights: Sobolev regularity estimates}
\begin{Lem}\label{lem:apri:weak}
Assume that $f \in L^r(\Omega) \setminus \{0\}$ is non-negative and $(r, \gamma) \in \mathcal{P}_{r, \gamma}$. Let $u_n$ be the weak solution of the problem $(P_n)$.
\begin{itemize}
    \item[(i)] For $r \in [1, \frac{N}{2})$ and $N > 2$, 
    \begin{equation*}
	\text{$u_n^{\frac{\mathfrak{S}_r+1}{2}}$ is uniformly bounded in $H_0^1(\Omega)$ with \ $\mathfrak{S}_r :=\frac{N(r-1) + \gamma r(N-2)}{N-2r}$}.
	\end{equation*}
	Moreover, $u_n$ is uniformly bounded in  $L^{\sigma_r}(\Omega)$ with $\sigma_r:= \frac{Nr(1+\gamma)}{N-2r}.$
    \item[(ii)] For $r= \frac{N}{2}$ and $N \geq 2$. Then 
    \begin{enumerate}
    \item   For $0 \leq \gamma \leq 1$ and $\beta >0 $ satisfying \eqref{range:est},
    $$\mathfrak{H} \left(\frac{u_n}{2}\right) \ \text{is uniformly bounded in} \  H_0^1(\Omega) \quad \text{where} \quad \mathfrak{H}(t):= \exp\left(\beta t\right)-1$$
and there exist constants  $C_1, C_2$ depending upon $\beta, S(N),f, |\Omega|$ but independent of $n$ such that  $$\int_{\Omega} \exp\left(\frac{ \beta N u_n}{N-2}\right) \leq C_1 \quad \text{when} \ N >2 \quad \text{and} \quad \|u_n\|_{L^\infty(\Omega)} \leq C_2 \quad  \text{when} \ N =2.$$
        \item  For $\tau \geq \gamma > 1$ and $\alpha >0 $ such that $\frac{\alpha}{2^{3-2\tau}} \max\{1, \alpha^\tau\}< \frac{1}{\tau S(N) \|f\|_{\frac{N}{2}}}$ 
    $$\mathfrak{D} \left(\frac{u_n}{2}\right) \ \text{is uniformly bounded in} \  H_0^1(\Omega) \quad \text{where} \quad \mathfrak{D}(t) = \left(\exp(\alpha t) -1\right)^{\tau}$$
and there exist constants  $C_3, C_4$ depending upon $\alpha, \tau, S(N),f, |\Omega|$ but independent of $n$ such that  $$\int_{\Omega} \exp\left(\frac{ 2N \alpha \tau u_n}{N-2}\right) \leq C_3 \quad  \text{when} \ N >2 \quad \text{and} \quad \|u_n\|_{L^\infty(\Omega)} \leq C_4 \quad  \text{when} \ N =2.$$
    \end{enumerate}
In particular, $u_n$ is uniformly bounded in $L^s(\Omega)$ for every $s \in [1, \infty)$ and $N >2.$
    \item[(iii)] For $r > \frac{N}{2}$ and $N \geq 2$, $u_n$ is uniformly bounded in  $L^\infty(\Omega).$
    \item[(iv)]  For $r \in [1, r^\sharp)$, $0 \leq \gamma < 1$ and $N > 2$, $u_n$ is uniformly bounded in $W_0^{1, q}(\Omega)$ with $q:= \frac{Nr(1+\gamma)}{N-r(1-\gamma)}$ and for any $r \in [1, \infty]$, $\gamma=1$ or $r \geq r^\sharp$ and $0 \leq \gamma < 1$, $u_n$ is uniformly bounded in $H_0^1(\Omega).$  
    \end{itemize}
\end{Lem}
\begin{proof}
Let $n \in \mathbb{N}$ and $u_n$ be the weak solution of the problem $(P_n)$ given by Lemma \ref{Lem:apriori}. To prove the uniform estimates for the sequence $\{u_n\}_{n \in \mathbb{N}}$, we divide the proof into four steps. \vspace{0.1cm}\\
\textbf{Step 1:} Since, $u_n \in L^\infty(\Omega) \cap H_0^1(\Omega)$ and positive, then for any $\epsilon>0$ and $\mathfrak{S} > 0$, $(u_n + \epsilon)^{\mathfrak{S}} - \epsilon^{\mathfrak{S}}$ belongs to $H_0^1(\Omega)$, therefore, an admissible test function in \eqref{def:notion:approx}. Taking it so for $\epsilon \in (0, \frac{1}{n})$ and $\mathfrak{S} \in [\gamma, \infty)$, it yields
\begin{equation}\label{est:apri-0}
\begin{split}
    \int_{\Omega} \nabla u_n \cdot \nabla (u_n + \epsilon)^{\mathfrak{S}} ~dx &+ \frac{C(N,s)}{2} \int_{\mathbb{R}^N} \int_{\mathbb{R}^N} \frac{(u_n(x)- u_n(y))((u_n + \epsilon)^{\mathfrak{S}}(x)- (u_n + \epsilon)^{\mathfrak{S}}(y))}{|x-y|^{N+2s}} ~dx ~dy \\
    &\leq \int_{\Omega} \frac{f_n(x)}{(u_n + \frac{1}{n})^\gamma} (u_n + \epsilon)^{\mathfrak{S}} ~dx \leq \int_{\Omega} f_n(x) (u_n + \epsilon)^{\mathfrak{S} - \gamma} ~dx.
\end{split}
\end{equation}
Passing $\epsilon \to 0$ in the above estimate via Fatou's theorem and using Lemma \ref{prelim:inequa} (i), we obtain
\begin{equation}\label{est:apri-1}
 \begin{split}
    \frac{4 \mathfrak{S}}{(\mathfrak{S} +1)^2}\int_{\Omega} |\nabla u_n^{\frac{\mathfrak{S} +1}{2}}|^2 ~dx &+ \frac{2 C(N,s) \mathfrak{S}}{(\mathfrak{S} +1)^2} \int_{\mathbb{R}^N} \int_{\mathbb{R}^N} \frac{((u_n^\frac{\mathfrak{S} +1}{2}(x)- u_n^\frac{\mathfrak{S} +1}{2}(y))^2}{|x-y|^{N+2s}} ~dx ~dy \\
    &\leq \int_{\Omega} f_n(x) u_n^{\mathfrak{S} - \gamma} ~dx.
\end{split}   
\end{equation}
In order to estimate the R.H.S. term of \eqref{est:apri-1}, we choose $\mathfrak{S} = \mathfrak{S}_r$ such that 
\begin{equation}\label{choice:1}
    \mathfrak{S}_1 = \gamma \quad \text{for} \ N \geq 2 \quad \text{and} \quad \frac{(\mathfrak{S}_r -\gamma)r}{(r-1)}= \frac{(\mathfrak{S}_r +1)}{2} \frac{2N}{N-2} \quad \text{for} \ 1 < r < \frac{N}{2}, \ N >2.
\end{equation}
By applying the H\"older inequality in view of the above choice of $\mathfrak{S}_r$ with $r>1$ and $N > 2$, we get
\begin{equation}\label{est:apri-2}
\begin{split}
 \int_{\Omega} f_n(x) u_n^{\mathfrak{S}_r - \gamma} ~dx &\leq \|f_n\|_{L^r(\Omega)} \left(\int_{\Omega} u_n^{(\mathfrak{S}_r - \gamma) r'} ~dx\right)^\frac{1}{r'} = \|f_n\|_{L^r(\Omega)} \left(\int_{\Omega} \left(u_n^{\frac{(\mathfrak{S}_r +1)}{2}}\right)^{\frac{2N}{N-2}} ~dx\right)^\frac{1}{r'}\\
 & \leq \|f_n\|_{L^r(\Omega)} \left( S(N) \int_{\Omega} |\nabla u_n^{\frac{\mathfrak{S}_r +1}{2}}|^2 ~dx \right)^{\frac{N}{r'(N-2)}}
\end{split}
\end{equation}
where the last inequality follows from Sobolev's embeddings and $S(N)$ denotes the best Sobolev constant. Combining \eqref{est:apri-1} and \eqref{est:apri-2}, we get
\begin{equation}\label{est:apri-3}
    \int_{\Omega} |\nabla u_n^{\frac{\mathfrak{S}_r +1}{2}}|^2 ~dx \leq  \left( S(N)\right)^{\frac{N(r-1)}{N-2r}} \left(\|f_n\|_{L^r(\Omega)}\frac{(\mathfrak{S}_r +1)^2}{4 \mathfrak{S}_r}\right)^\frac{(N-2)r}{N-2r} \leq C(\|f\|_{L^r(\Omega)}, \mathfrak{S}_r, N).
\end{equation}
The above estimates in the case $r=1$ and $N \geq 2$ holds trivially. 
An easy computation with the above choice of $\mathfrak{S}_r$ implies
\begin{equation}\label{choice:2}
    \frac{N(\mathfrak{S}_r +1)}{N-2} = \frac{Nr(1+\gamma)}{N-2r} = \sigma_r
\end{equation}
Then, by using the uniform estimates in \eqref{est:apri-3} and Sobolev's embeddings, we obtain
\[\{u_n\}_{n \in \mathbb{N}} \ \text{is uniformly bounded in} \ L^{\sigma_r}(\Omega) \ \text{when $1 \leq r < \frac{N}{2},\ N > 2$}.\]
\textbf{Step 2:} Let $r= \frac{N}{2}$ and $N \geq 2$. In case of weak singularity {\it i.e.} $0 \leq \gamma \leq 1$, let us consider an increasing, convex and locally Lipschitz function $\mathfrak{H}: \mathbb{R} \to \mathbb{R}$ defined as
$$\mathfrak{H}(t):= \exp\left(\beta t\right)-1$$
where $\beta >0$ whose exact choice will be determined later. Since $u_n \in L^\infty(\Omega) \cap H_0^1(\Omega)$ for every $n \in \mathbb{N}$, we get $\mathfrak{H}(u_n) \in H_0^1(\Omega).$ A simple computations leads to
\begin{equation}\label{est:expo}
 \int_{\Omega} \nabla u_n \cdot \nabla \mathfrak{H}(u_n) ~dx= \int_{\Omega} \mathfrak{H}'(u_n) |\nabla u_n|^2  ~dx= \frac{4}{\beta} \int_{\Omega} \left| \nabla \mathfrak{H} \left(\frac{u_n}{2}\right)\right|^2  ~dx   
\end{equation}
Now, by testing the energy formulation \eqref{def:notion:approx} with $\mathfrak{H}(u_n)$ and using \eqref{est:expo}, we obtain
\begin{equation}\label{imp:est}
    \begin{split}
       \frac{4}{\beta} &\int_{\Omega} \left| \nabla \mathfrak{H} \left(\frac{u_n}{2}\right)\right|^2 ~dx  \\
       &\leq  \int_{\Omega} \nabla u_n \cdot \nabla \mathfrak{H}(u_n) ~dx + \frac{C(N,s)}{2} \int_{\mathbb{R}^N} \int_{\mathbb{R}^N} \frac{(u_n(x)- u_n(y))(\mathfrak{H}(u_n)(x) - \mathfrak{H}(u_n)(y))}{|x-y|^{N+2s}} ~dx ~dy \\
        &\leq \int_{\Omega} \frac{f_n(x)}{(u_n + \frac{1}{n})^\gamma} \mathfrak{H}(u_n) ~dx \leq  \int_{\Omega} \frac{f(x) \mathfrak{H}(u_n)}{u_n^\gamma}  ~dx = \int_{\Omega} \frac{f(x) \left(\exp(\beta u_n)-1\right)}{u_n^\gamma}  ~dx
    \end{split}
\end{equation}
\textit{Case 1:} $0 < \gamma \leq 1.$ \vspace{0.1cm}\\ To estimate the last term in \eqref{imp:est}, we use Lemma \ref{prelim:inequa}(iii) and $\exp(t)-1 \leq t \exp(t)$ for every $t \geq 0$,
\begin{equation*}
    \begin{split}
       \int_{\Omega} \left| \nabla \mathfrak{H} \left(\frac{u_n}{2}\right)\right|^2 ~dx  &\leq \frac{\beta}{4} \left(\beta\gamma^{-1}\right)^\gamma\int_{\Omega \cap \{u_n \leq 1\}} f(x) \left( \frac{\exp\left(\beta\gamma^{-1} u_n\right) -1}{ \beta\gamma^{-1} u_n}\right)^\gamma  ~dx \\
       & \quad \quad + \frac{\beta}{4} \int_{\Omega \cap \{u_n \geq 1\}} f(x) \exp\left(\beta u_n\right)  ~dx\\
        & \leq C^\sharp \int_{\Omega} f(x) \left(\exp\left(\frac{\beta u_n}{2}\right)\right)^2  ~dx = C^\sharp \int_{\Omega} f(x) \left(\mathfrak{H} \left(\frac{u_n}{2}\right) +1\right)^2  ~dx
\end{split}
\end{equation*}
where $C^\sharp:= \frac{\beta}{4} \max\{1, \left(\beta\gamma^{-1}\right)^\gamma \}.$ \vspace{0.1cm}\\
\textit{Case 2:} $\gamma=0.$ \vspace{0.1cm}\\
In this case the estimate in \eqref{imp:est}  takes the following form 
\begin{equation*}
    \begin{split}
       \int_{\Omega} \left| \nabla \mathfrak{H} \left(\frac{u_n}{2}\right)\right|^2 ~dx  & \leq C^\sharp \int_{\Omega} f(x) \left(\exp\left(\frac{\beta u_n}{2}\right)\right)^2  ~dx = C^\sharp \int_{\Omega} f(x) \left(\mathfrak{H} \left(\frac{u_n}{2}\right) +1\right)^2  ~dx
\end{split}
\end{equation*}
with $C^\sharp=\frac{\beta}{4}.$ Now, by using the fact that $(\mathfrak{H}(t)+1)^2 \leq 2 (\mathfrak{H}(t))^2+1)$ and H\"older inequality with exponents $\frac{N}{2}$ and $\frac{N}{N-2}$ , we further estimate 
\begin{equation*}
  \begin{split}
       \int_{\Omega} \left| \nabla \mathfrak{H} \left(\frac{u_n}{2}\right)\right|^2~dx &  \leq  2 C^\sharp \left(\int_{\Omega} f(x) ~dx + \int_{\Omega} f(x) \left(\mathfrak{H} \left(\frac{u_n}{2}\right)\right)^2 ~dx \right)\\
         & \leq 2 C^\sharp \left( \|f\|_1 + \|f\|_{\frac{N}{2}} \left\| \mathfrak{H} \left(\frac{u_n}{2}\right)\right\|_{\frac{2N}{N-2}}^2 \right).
    \end{split}
\end{equation*}
Now, by choosing $\beta$ small enough such that
\[
	\frac{2}{S(N) \|f\|_{\frac{N}{2}}} \geq
	 \left\{
	\begin{array}{ll}
	\beta \max\{1, \left(\beta\gamma^{-1}\right)^\gamma\} & \text{ if } \ 0 < \gamma \leq 1, \vspace{0.1cm}\\
	\beta & \text{ if } \ \gamma=0,
	\end{array} 
	\right.
	\]
and using Sobolev's embeddings and $|\Omega| < \infty$, we obtain
\begin{equation*}
    \left\|\nabla \mathfrak{H} \left(\frac{u_n}{2}\right)\right\|_2^2 \leq \frac{2 C^\sharp \|f\|_1}{\left(1 - 2 C^\sharp S(N) \|f\|_{\frac{N}{2}}\right)}
\end{equation*}
and 
\begin{equation*}
    \int_{\Omega} \exp \left(\frac{ \beta N u_n}{N-2}\right)~dx \leq C
\end{equation*}
where $C$ depends upon $\beta, S(N), \|f\|_{\frac{N}{2}}, |\Omega|$ but independent of $n.$ 

For the case of strong singularity {\it i.e.} $\gamma>1$, we consider the following locally Lipschitz and increasing function $\mathfrak{D} : \mathbb{R} \to \mathbb{R}^+ \cup \{0\}$ defined as
$$\mathfrak{D}(t) = \left(\exp(\alpha t) -1\right)^{\tau}$$
where $\tau \geq \gamma >1$ and $\alpha >0$ whose exact choices will be highlighted later. The regularity properties of the sequence $u_n$ and Lemma \ref{prelim:inequa}(iii) yields $\mathfrak{D}(u_n) \in H_0^1(\Omega)$ and 
\begin{equation}\label{est:expo-1}
\begin{split}
 \int_{\Omega} \nabla u_n \cdot \nabla \mathfrak{D}(u_n) ~dx &= \int_{\Omega} \mathfrak{D}'(u_n) |\nabla u_n|^2  ~dx= \alpha \tau \int_{\Omega} \left| \exp\left(\frac{\alpha u_n}{2}\right) \left(\exp(\alpha u_n) -1 \right)^{\frac{\tau-1}{2}} \nabla u_n \right|^2  ~dx \\
 & \geq \alpha \tau \int_{\Omega} \left| \exp\left(\frac{\alpha u_n}{2}\right) \left(\exp\left(\frac{\alpha u_n}{2}\right) -1 \right)^{\tau-1} \nabla u_n \right|^2  ~dx \\
 & = \frac{4}{\alpha \tau} \int_{\Omega} \left|\nabla  \left(\exp\left(\frac{\alpha u_n}{2}\right) -1 \right)^{\tau} \right|^2  ~dx = \frac{4}{\alpha \tau} \int_{\Omega} \left|\nabla  \mathfrak{D}\left(\frac{u_n}{2}\right) \right|^2  ~dx
 \end{split}
\end{equation}
 By taking $\mathfrak{D}(u_n)$ as a test function in the equation \eqref{def:notion:approx} and using \eqref{est:expo-1}, we obtain
\begin{equation*}
    \begin{split}
       \frac{4}{\alpha \tau} & \int_{\Omega} \left|\nabla  \mathfrak{D}\left(\frac{u_n}{2}\right) \right|^2  ~dx \\ &\leq  \int_{\Omega} \nabla u_n \cdot \nabla \mathfrak{D}(u_n) ~dx + \frac{C(N,s)}{2} \int_{\mathbb{R}^N} \int_{\mathbb{R}^N} \frac{(u_n(x)- u_n(y))(\mathfrak{D}(u_n)(x) - \mathfrak{D}(u_n)(y))}{|x-y|^{N+2s}} ~dx ~dy \\
        &\leq \int_{\Omega} \frac{f_n(x)}{(u_n + \frac{1}{n})^\gamma} \mathfrak{D}(u_n) ~dx \leq  \int_{\Omega} \frac{f(x) \mathfrak{D}(u_n)}{u_n^\gamma}  ~dx = \int_{\Omega} \frac{f(x) \left(\exp(\alpha u_n)-1\right)^\tau}{u_n^\gamma}  ~dx \\
        &\leq \alpha^\tau \int_{\Omega \cap \{u_n \leq 1\}} f(x) \left( \frac{\exp\left(\alpha u_n\right) -1}{ \alpha u_n}\right)^\tau  ~dx + \int_{\Omega \cap \{u_n \geq 1\}} f(x) \exp\left(\alpha \tau u_n\right)  ~dx\\
        & \leq \max\{1, \alpha^\tau\} \int_{\Omega } f(x) \exp\left(\alpha \tau u_n\right)  ~dx
    \end{split}
\end{equation*}
where to write the last two inequalities we have used $t^\tau \leq t^\gamma$ for $t \in (0,1]$ and $\exp(t)-1\leq t \exp(t)$ for every $t \geq 0.$ Now, by using the fact that $\exp(2\tau t) = \left(\mathfrak{D}^\frac{1}{\tau}(t) +1\right)^{2\tau} \leq 2^{2\tau-1} (\mathfrak{D}(t))^2+1)$ and H\"older inequality with exponents $\frac{N}{2}$ and $\frac{N}{N-2}$, we obtain
\begin{equation*}
  \begin{split}
       \int_{\Omega} \left| \nabla \mathfrak{D} \left(\frac{u_n}{2}\right)\right|^2~dx &  \leq C^* \int_{\Omega} f(x) \left(\exp\left(\frac{\alpha u_n}{2}\right)\right)^{2\tau}  ~dx \leq 2^{2\tau-1} C^* \int_{\Omega} f(x) \left(\mathfrak{D}^\frac{1}{\tau} \left(\frac{u_n}{2}\right) +1\right)^{2\tau}  ~dx\\
       &\leq 2^{2\tau-1} C^* \left(\int_{\Omega} f(x) ~dx + \int_{\Omega} f(x) \left(\mathfrak{D} \left(\frac{u_n}{2}\right)\right)^2 ~dx \right)\\
         & \leq 2^{2\tau-1} C^* \left( \|f\|_1 + \|f\|_{\frac{N}{2}} \left\| \mathfrak{D} \left(\frac{u_n}{2}\right)\right\|_{\frac{2N}{N-2}}^2 \right) \quad \quad C^*:= \frac{\alpha \tau}{4} \max\{1, \alpha^\tau\}.
    \end{split}
\end{equation*}
Now, by choosing $\alpha$ small enough such that $\alpha \max\{1, \alpha^\tau\}< \frac{2^{3-2\tau}}{\tau S(N) \|f\|_{\frac{N}{2}}}$ and using Sobolev's embeddings, we obtain
\begin{equation*}
    \left\|\nabla \mathfrak{D} \left(\frac{u_n}{2}\right)\right\|_2^2 \leq \frac{2^{2\tau -1} C^* \|f\|_1}{\left(1 - 2^{2\tau -1} C^* S(N) \|f\|_{\frac{N}{2}}\right)} \quad \text{and} \quad
    \int_{\Omega} \exp \left(\frac{2N \alpha \tau  u_n}{N-2}\right)~dx \leq C
\end{equation*}
where $C$ depends upon $\alpha, \tau, S(N), \|f\|_{\frac{N}{2}}, |\Omega|$ but independent of $n.$
\vspace{0.1cm}\\
\textbf{Step 3:}
To prove the boundedness result when $r > \frac{N}{2}$, we use the classical arguments from the seminal paper of Stampacchia \cite{Stampacchia}. Choosing $G_k(u_n):= (u_n-k)^+ \in H_0^1(\Omega)$ with $k \geq 1$ as a test function in the energy formulation \eqref{def:notion:approx}. Then, by using Sobolev embeddings, the H\"older inequality, $f_n \leq f$ and Lemma \ref{Lem:apriori}, we get
\begin{equation}\label{est:apri-10}
    \begin{split}
        \left(\int_{A_k} |G_k(u_n)|^{\frac{2N}{N-2}} ~dx \right)^\frac{N-2}{N} &\leq \int_{A_k} |\nabla G_k(u_n)|^2 ~dx = \int_{\Omega} \nabla u_n \cdot \nabla G_k(u_n) ~dx\\
        &\leq \int_{\Omega} \frac{f_n}{\left(u_n + \frac{1}{n}\right)^\gamma} G_k(u_n) ~dx \leq \int_{A_k} f_n(x) G_k(u_n) ~dx\\
        & \leq C \|f\|_{L^r(\Omega)} \left( \int_{A_k} (G_k(u_n))^{\frac{2N}{N-2}}\right)^\frac{N-2}{2N} |A_k|^{1-\frac{N-2}{2N}- \frac{1}{r}}
        \end{split}
\end{equation}
where $A_k:=\{x \in \Omega: u_n(x) \geq k\}.$ Let $h>k \geq 1$, then $A_h \subset A_k$ and $G_k(u_n) \geq h-k$ for $x \in A_h.$ Now, by manipulating the estimate in \eqref{est:apri-10} with above facts, we obtain
\begin{equation*}
    \begin{split}
     |h-k| |A_h|^\frac{N-2}{2N} &\leq \left(\int_{A_h} |G_k(u_n)|^{\frac{2N}{N-2}} ~dx \right)^\frac{N-2}{2N} \leq \left(\int_{A_k} |G_k(u_n)|^{\frac{2N}{N-2}} ~dx \right)^\frac{N-2}{2N} \\
     &\leq C \|f\|_{L^r(\Omega)}  |A_k|^{1-\frac{N-2}{2N}- \frac{1}{r}}
    \end{split}
\end{equation*}
which further implies
$$|A(h)| \leq C \frac{\|f\|_{L^r(\Omega)}^\frac{2N}{(N-2)} |A_k|^{\frac{2N}{N-2}\left(1-\frac{N-2}{2N}- \frac{1}{r}\right)}}{|h-k|^\frac{2N}{N-2}}.$$
Since, $r > \frac{N}{2}$, we have that
$$ \frac{2N}{N-2}\left(1-\frac{N-2}{2N}- \frac{1}{r}\right) >1.$$
Hence, we apply Lemma \ref{lem:iter} with the choice of $\psi(k)= |A_k|$, consequently there exists $k_0$ such that $\psi(k)=0$ for all $k \geq k_0$ and thus our claim. \vspace{0.1cm}\\
\textbf{Step 4:} Let $r \in [1, r^\sharp)$ and $q<2$ defined as in the statement of theorem. Observe that for any $r \in [1, r^\sharp)$, $\mathfrak{S}_r \in [\gamma,1)$ and 
\begin{equation}\label{cont:expo}
    {\frac{(1-\mathfrak{S}_r)q}{2-q}} = \frac{Nr (1+\gamma)}{N-2r}=\sigma_r.
\end{equation}
By applying H\"older inequality with exponent $\frac{2}{q}$ and $\frac{2}{2-q}$ and using claim in \textbf{Step 1}, we get
\begin{equation*}
    \begin{split}
    \int_{\Omega} |\nabla u_n|^q ~dx &= \int_{\Omega} \frac{|\nabla u_n|^q}{u_n^{\frac{(1-\mathfrak{S}_r)q}{2}}}  u_n^{\frac{(1-\mathfrak{S}_r)q}{2}}  ~dx \\
    &\leq \left(\int_{\Omega} u_n^{\mathfrak{S}_r-1} |\nabla u_n|^2 ~dx \right)^\frac{q}{2}  \left(\int_{\Omega} u_n^{\frac{(1-\mathfrak{S}_r)q}{2-q}} ~dx \right)^\frac{2-q}{2}\\
    & \leq \left(\frac{4}{(\mathfrak{S}_r+1)^2}\int_{\Omega} |\nabla u_n^{\frac{\mathfrak{S}_r+1}{2}}|^2 ~dx \right)^\frac{q}{2}  \left(\int_{\Omega} u_n^{\sigma_r}\right)^\frac{2-q}{2} \leq C \ \text{(independent of $n$)}.
    \end{split} 
\end{equation*}
As a special case, when $r \in [1, \infty]$, $\gamma=1$ or $r \geq r^\sharp$ and $0 \leq \gamma < 1$, the proof of Claim $(i)$ can repeated by taking $\mathfrak{S}_r=1$ and which further implies that $u_n$ is uniformly bounded in $H_0^1(\Omega).$
\end{proof}
\begin{remark}
\begin{itemize}
    \item From Lemma \ref{lem:apri:weak}(i), (iv) and Lemma \ref{Lem:apriori}(iii) we observe that for any $\gamma \geq 1$, $r \geq 1$, and $0 \leq \gamma <1$, $r \geq r^\sharp$, $u_n$ is uniformly bounded in $H^1_{loc}(\Omega)$ since $\mathfrak{S}_r \geq 1.$
    
    \item From \eqref{cont:expo}, we observe that for any $0 \leq \gamma <1$
$$q= \frac{2\sigma}{1-\mathfrak{S}_r + \sigma}$$ which implies that there is a kind of $``$continuity$"$ in the summability exponent $q$. Precisely, as $r \to r^\sharp$, $\mathfrak{S}_r \to 1$ and $q \to 2.$
\end{itemize}
\end{remark}
\begin{Lem}\label{lem:lower:reg}
Assume that $f \in L^r(\Omega) \setminus \{0\}$ is non-negative and $(r, \gamma) \in \mathcal{P}_{r, \gamma}$. Let $u_n$ be the weak solution of the problem $(P_n)$. Then, 
    \begin{equation*}
	u_n^{\frac{\mathfrak{S}+1}{2}} \ \text{is uniformly bounded in} \ H_0^1(\Omega) \ \text{with} \ \mathfrak{S} \in \left\{
	\begin{array}{ll}
	(0, \gamma] & \text{ if } \ \gamma+\frac{1}{r} < 1, \vspace{0.1cm}\\
	\left(\gamma-1+\frac{1}{r}, \gamma \right] & \text{ if } \ \gamma+\frac{1}{r} \geq 1.
	\end{array} 
	\right.
	\end{equation*}
\end{Lem}
\begin{proof}
Taking $(u_n + \epsilon)^{\mathfrak{S}} - \epsilon^{\mathfrak{S}}$ as a test function in \eqref{def:notion:approx}  for $\epsilon \in (0, \frac{1}{n})$ and $\mathfrak{S} \in (0,\gamma]$  and using lower boundary estimates in Theorem \ref{Green:thm:1}, we obtain
\begin{equation}\label{low:est:apri-0}
\begin{split}
    \int_{\Omega} \nabla u_n \cdot \nabla (u_n + \epsilon)^{\mathfrak{S}} ~dx &+ \frac{C(N,s)}{2} \int_{\mathbb{R}^N} \int_{\mathbb{R}^N} \frac{(u_n(x)- u_n(y))((u_n + \epsilon)^{\mathfrak{S}}(x)- (u_n + \epsilon)^{\mathfrak{S}}(y))}{|x-y|^{N+2s}} ~dx ~dy \\
    & \leq \int_{\Omega} \frac{f_n(x)}{(u_n + \epsilon)^{\gamma-\mathfrak{S}}} ~dx \leq C \int_{\Omega} \frac{f(x)}{\delta^{\gamma-\mathfrak{S}}(x)} ~dx \\
    & \leq C \|f\|_{L^r(\Omega)} \left(\int_{\Omega} \delta^{\frac{-r(\gamma-\mathfrak{S})}{r-1}}(x) ~dx\right)^{\frac{r-1}{r}} \\
    &\leq C(\|f\|_{L^r(\Omega)}, N, \Omega,r) \quad \text{if} \quad \frac{-r(\gamma-\mathfrak{S})}{r-1} >-1 \ \text{and} \ \mathfrak{S} >0
\end{split}
\end{equation}
where the last condition is equivalent to
\[\frac{-r(\gamma-\mathfrak{S})}{r-1} >-1 \ \text{and} \ \mathfrak{S} >0 \quad \text{if and only if} \quad \mathfrak{S} \in \left\{
	\begin{array}{ll}
	(0, \gamma] & \text{ if } \ \gamma+\frac{1}{r} < 1, \vspace{0.1cm}\\
	\left(\gamma-1+\frac{1}{r}, \gamma \right] & \text{ if } \ \gamma+\frac{1}{r} \geq 1.
	\end{array} 
	\right.
	\]
Now, by using the same arguments as in proof of Lemma \ref{lem:apri:weak}(i) and passing $\epsilon \to 0$, we obtain our claim.	
\end{proof}
\subsection{Green's function estimates}
In this section, we prove series of Lemma involving the upper and lower estimate of the action of Green operator on the logarithmic perturbation of the distance function. Similar type of Green estimates for large class of nonlocal operator is proved in \cite{Arora_Tai}. For estimates near the boundary {\it i.e.} in $\Omega_\eta$, we partition the set into the following five components:
For this, we partition the set $\Omega_\eta$ into the following five components 
	$$\Omega_1:= B(x, \delta(x)/2), \quad \Omega_2:= \Omega_{\eta} \setminus B(x,1),$$
	$$\Omega_3:= \{y: \delta(y) < \delta(x)/2\} \cap B(x,1), \quad \Omega_4:= \left\{y: \frac{3 \delta(x)}{2} < \delta(y) < \eta \right\} \cap B(x,1),$$
	$$\Omega_5:= \left\{y: \frac{\delta(x)}{2} < \delta(y) < \frac{3 \delta(x)}{2}\right\} \cap \left( B(x,1) \setminus B(x, \delta(x)/2)\right),$$
and set $\ell(t)=\ln\left(\frac{\mathcal{D}_\Omega}{t}\right).$
For $x \in \Omega_\eta$, we denote $\phi_x: B(x,1) \to B(0,1)$ be a diffeomorphism such that
	\begin{align*}\phi_x(\Omega \cap B(x,1)) = B(0,1) \cap \{y \in \mathbb{R}^N : y \cdot e_N >0\}, \\ 
	\phi_x(y) \cdot e_N = \delta(y) \ \text{for} \ y \in B(x,1)\  \text{and} \ \phi_x(x)= \delta(x) e_N.
	\end{align*}
\begin{Lem}\label{lem:greest1}
For $\Xi \in [0,1)$, we have
	$$\mathbb{G}^\Omega\left[ \frac{\ell^{-\Xi}(\delta(\cdot))}{\delta(\cdot)}\right](x) \gtrsim \delta(x) \ell^{1-\Xi}(\delta(x)) \quad \forall \ x \in \Omega.$$
\end{Lem}
\begin{proof} Let $\eta>0$ small. We begin by splitting the integrals over two regions $\Omega_\eta$ and $\Omega \setminus \Omega_\eta$ as follows
	\begin{equation} \label{eq:GI1I2}\begin{aligned} \mathbb{G}^\Omega\left[ \frac{\ell^{-\Xi}(\delta(\cdot))}{\delta(\cdot)} \right](x) &= \sum_{i=1}^5 \int_{\Omega_i} \frac{G^{\Omega}(x,y)}{\delta(y)} \ell^{-\Xi}(\delta(y)) ~dy  + \int_{\Omega \setminus \Omega_{\eta}} \frac{G^{\Omega}(x,y)}{\delta(y)} \ell^{-\Xi}(\delta(y)) ~dy \\
	&:= \sum_{i=1}^5 I_1^{(i)}(x) + I_2(x).
	\end{aligned} \end{equation}
As we know that $I_1^{(i)}(x), I_2(x) \geq 0$ for every $x \in \Omega$ and $i=1,2,\dots,5$, so it is enough find the lower estimate of the term $I_1^{(4)}$ when $x \in \Omega_\frac{\eta}{2}$, and the lower estimate of the term $I_2$ when $x \in \Omega \setminus \Omega_\frac{\eta}{2}.$ For $y \in \Omega_4$ and $x \in \Omega_\frac{\eta}{2}$, we have
	$$\left( \frac{\delta(x) \delta(y)}{|x-y|^{2}} \wedge 1 \right) \asymp \frac{\delta(x) \delta(y)}{|x-y|^{2}},\quad \ell^{-\Xi}(\delta(x)) \leq \ell^{-\Xi}(\delta(y)) .$$
	Therefore
	\begin{equation} \label{est:I1-1} \begin{aligned}
	I_1(x) \geq \int_{\Omega_4}  \frac{G^{\Omega}(x,y)}{\delta(y)} \ell^{-\Xi}(\delta(y)) ~dy   
	\geq  \ell^{-\Xi}(\delta(x)) \int_{\Omega_4} \frac{G^\Omega(x,y)}{\delta(y)} ~dy := \ell^{-\Xi}(\delta(x)) \mathcal{J}
	\end{aligned} \end{equation}
	Now, by using estimates on \cite[Proof of Lemma 3.3, Page 40]{AbaGomVaz_2019} and performing change of variables via diffeomorphism $\phi_x$, we obtain
	\begin{equation}\label{est:I1-2} 
	\begin{split}
	 &\mathcal{J} \gtrsim  \delta(x) \int_{\{\frac{3\delta(x)}{2} < w_N < \eta\} \cap B(0,1)} \frac{1}{(|\delta(x)-w_N| + |w'|)^{N}} ~dw_N dw' \\
	 &= \delta(x) \int_{3/2}^{\eta/\delta(x)}  \int_0^{1/\delta(x)} \frac{t^{N-2} }{(|1-h|+t)^{N}} ~dt~dh \\
	&=  \delta(x) \int_{3/2}^{\eta/\delta(x)} \frac{1}{(h-1)} \int_0^{1/(h-1)\delta(x)} \frac{ r^{N-2} }{(1+r)^{N}} ~dr~dh \\
	& \gtrsim \delta(x) \int_{3/2}^{\eta/\delta(x)} \frac{1}{(h-1)} \int_1^{1/(h-1)\delta(x)} \frac{1 }{(1+r)^{2} } ~dr~dh \gtrsim \delta(x) \int_{3/2}^{\eta/\delta(x)} \frac{1}{(h-1)} ~dh \\
	&\gtrsim  \delta(x) \int_{3/2}^{\eta/\delta(x)}  \frac{1}{h} ~dh  = \delta(x) \left(\ln\left(\frac{\eta}{\delta(x)}\right) - \ln\left(\frac{3}{2} \right) \right).
	\end{split}
	\end{equation}
By combining \eqref{eq:GI1I2}, \eqref{est:I1-1} and \eqref{est:I1-2}, we obtain
\begin{equation}\label{lower:final1}
    \mathbb{G}^\Omega\left[ \frac{1}{\delta(\cdot)} \ell^{-\Xi}(\delta(\cdot)) \right](x) \geq I_1(x) \geq C \delta(x) \ell^{1-\Xi}(\delta(x)) \ \text{for} \ x \in \Omega_\frac{\eta}{2}
\end{equation}
where the constant is independent of the parameter $\Xi.$ Let $ x \in \Omega \setminus \Omega_\frac{\eta}{2}>0.$ 	Since the operator $\mathbb{G}^\Omega$ maps $ L^\infty_c(\Omega) \to \delta^\gamma C(\overline{\Omega})$ (\cite[Theorem 2.10]{AbaGomVaz_2019}), therefore we have the following estimates:
	\begin{equation}\label{lower:final2}
	\begin{split}
	 I_2(x)  & = \int_{\Omega \setminus \Omega_{\eta}} \frac{G^{\Omega}(x,y)}{\delta(y)} \ell^{-\Xi}(\delta(y)) ~dy \geq  \ell^{-\Xi}(\eta) \mathbb{G}^\Omega\left[ \frac{\chi_{\Omega \setminus \Omega_{\eta}}}{\delta}\right](x)\\
	& \geq C \ell^{-\Xi}\left(\frac{\eta}{2}\right) \delta(x) \geq C(\eta) \ell^{1-\Xi}\left(\frac{\eta}{2}\right) \delta(x) \geq C(\eta) \ell^{1-\Xi}\left(\delta(x)\right) \delta(x).
	\end{split}  
	\end{equation}
Finally, by combining \eqref{lower:final1} and \eqref{lower:final2}, we obtain our claim.		
\end{proof}

\begin{Lem}\label{lem:greest2}
For $\Xi \in (0,1)$ following estimate holds:
	$$\mathbb{G}^\Omega\left[ \frac{1}{\delta(\cdot)} \ell^{-\Xi}(\delta(\cdot)) \right](x) \lesssim \delta(x) \ell^{1-\Xi}(\delta(x)) \quad \forall \ x \in \Omega.$$
\end{Lem}
\begin{proof} 
Let $I_1(x)$ and $I_2(x)$ as in \eqref{eq:GI1I2}. To derive the estimate for $I_1$, we divide the proof into two cases depending upon the location of the point $x.$  \vspace{0.1cm}\\
\textbf{Case 1:} $x \in \Omega_{\frac{\eta}{2}}.$ \vspace{0.1cm}\\ By partitioning the domain of integral $I_1(x)$ over $\{\Omega_i\}_{i=1}^5$, we find the upper estimates over each subdomain $\Omega_i$.  Observing, for $y_1 \in \Omega_1$ and $y_2 \in \cup_{i=2}^5 \Omega_i$, we have $$\left( \frac{\delta(x) \delta(y_1)}{|x-y_1|^{2}} \wedge 1 \right) \asymp 1\ \text{and} \ \left( \frac{\delta(x) \delta(y_2)}{|x-y_2|^{2}} \wedge 1 \right) \asymp \frac{\delta(x) \delta(y_2)}{|x-y_2|^{2}}.$$
	Now, by again using the change of variables via diffeomorphism $\phi_x$ and \cite[Proof of Lemma 3.3]{AbaGomVaz_2019}, we get estimates in each domain. \vspace{0.1cm}\\
\textit{Estimate over $\Omega_1$:} Choosing $\eta$ small enough such that $0< \eta < \frac{2\mathcal{D}_{\Omega}}{\exp(1)}$, for some  $c \in (0,1)$, we have, for any $y \in \Omega_1$, 
	$$ \delta(y) \leq \frac{3}{2} \delta(x), \quad \ell^{-\Xi}(\delta(y)) \leq c^{-\Xi} \ell^{-\Xi}(\delta(x)) \leq c^{-1} \ell^{-\Xi}(\delta(x)). $$
	Therefore
	\begin{equation*}
	\begin{split}
	\int_{\Omega_1}  \frac{G^{\Omega}(x,y)}{\delta(y)} \ell^{-\Xi}(\delta(y)) ~dy &\leq c^{-1} \ell^{-\Xi}(\delta(x)) \frac{1}{\delta(x)} \int_{B(x, \delta(x)/2)} \frac{1}{|x-y|^{N-2}} ~dy \\
	& \leq c^{-1} \ell^{-\Xi}(\delta(x)) \delta(x) \leq C \ell^{1-\Xi}(\delta(x)) \delta(x),
	\end{split}
	\end{equation*}
	where $C$ is independent of parameter $\Xi.$ \vspace{.05cm}\\ 
\textit{Estimate over $\Omega_2$:}
	\begin{equation*}
	\begin{split}
	\int_{\Omega_2} \frac{G^{\Omega}(x,y)}{\delta(y)} \ell^{-\Xi}(\delta(y)) ~dy &\leq \ell^{-\Xi}\left(\eta\right) \delta(x) \int_{\Omega_2} \frac{1}{|x-y|^{n}}  ~dy \leq C(\eta) \  \ell^{1-\Xi}(\delta(x)) \delta(x),
	\end{split}
	\end{equation*}
	where $C$ is independent of parameter $\Xi.$ \vspace{.05cm}\\
\textit{Estimate over $\Omega_3$:} We note that, for any $y \in \Omega_3$,  $\ell^{-\Xi}(\delta(y)) \leq \ell^{-\Xi}(\delta(x)).$ 
	Therefore
	\begin{equation*}
	\begin{split}
	\int_{\Omega_3}  & \frac{G^{\Omega}(x,y)}{\delta(y)} \ell^{-\Xi}(\delta(y)) ~dy \leq \ell^{-\Xi}(\delta(x)) \delta(x) \int_{\Omega_3} \frac{1}{|x-y|^{N}} ~dy \\
	& \lesssim  \ell^{-\Xi}(\delta(x)) \delta(x) \int_{|w'| <1} \int_0^{\delta(x)/2} \frac{1}{(|\delta(x)-w_N| + |w'|)^{N}} ~dw_N dw'\\
	& \lesssim  \ell^{-\Xi}(\delta(x)) \delta(x) \int_0^{1/\delta(x)} \int_0^{1/2} \frac{h^{N-2} }{((1-t)+h)^{N}} ~dt~dh\\
	& \lesssim  \ell^{-\Xi}(\delta(x)) \delta(x) \int_0^{1/\delta(x)} \frac{h^{N-2}}{(1+h)^{N}} ~dh \leq  C(\eta) \ell^{1-\Xi}(\delta(x)) \delta(x).
	\end{split}
	\end{equation*}
\textit{Estimate over $\Omega_4$:} We have
	\begin{equation*}
	\begin{split}
	\int_{\Omega_4} &  \frac{G^{\Omega}(x,y)}{\delta(y)} \ell^{-\Xi}(\delta(y)) ~dy  \leq \delta(x) \int_{\Omega_4} \frac{1}{|x-y|^{N} \ell^{\Xi}(\delta(y))} ~dy \\
	& \lesssim \delta(x) \int_{\{\frac{3\delta(x)}{2} < w_N < \eta\} \cap B(0,1)} \frac{1}{(|\delta(x)-w_N| + |w'|)^{N} \ell^{\Xi}(w_N)} ~dw_N dw'\\
	& = \delta(x) \int_{3/2}^{\eta/\delta(x)}  \int_0^{1/\delta(x)} \frac{t^{N-2}}{(|1-h|+t)^{N} \ell^{\Xi}(h\delta(x)) } ~dt~dh\\
	& =  \delta(x) \int_{3/2}^{\eta/\delta(x)} \frac{1}{(h-1) \ell^{\Xi}(h\delta(x)) } \int_0^{1/(h-1)\delta(x)} \frac{ r^{N-2} }{(1+r)^{N} } ~dr~dh \\
	& \lesssim  \delta(x) \int_{3/2}^{\eta/\delta(x)} \frac{1}{(h-1) \ell^{\Xi}(h\delta(x)) } \int_0^{1/(h-1)\delta(x)} \frac{1 }{(1+r)^{2} } ~dr~dh \\
	& \lesssim \delta(x) \int_{3/2}^{\eta/\delta(x)} \frac{1}{(h-1) \ell^{\Xi}(h\delta(x)) } ~dh \lesssim \delta(x) \int_{3/2}^{\eta/\delta(x)}  \frac{ \ell^{-\Xi}(h\delta(x)) }{h} ~dh \\
	& \lesssim \delta(x) \int^{\ln(2\mathcal{D}_{\Omega}/3\delta(x))}_{\ln(\mathcal{D}_{\Omega}/\eta)} \frac{1}{t^\Xi} dt  \leq \frac{C(\eta)}{(1-\Xi)} \ell^{1-\Xi}(\delta(x)) \delta(x) 
	\end{split}
	\end{equation*}
\textit{Estimate over $\Omega_5$:} Again, by choosing $\eta$ small enough such that $ \eta < \frac{2A}{\exp(1)}$,  for some $c \in (0,1)$, we have, for $y \in \Omega_5$,  
	$$ \delta(y) \leq \frac{3}{2} \delta(x) \ \text{and} \ \ell^{-\Xi}(\delta(y)) \leq c^{-\Xi} \ell^{-\Xi}(\delta(x)) \leq c^{-1} \ell^{-\Xi}(\delta(x)).$$
	Therefore
	\begin{equation*}
	\begin{split}
	\int_{\Omega_5}  &\frac{G^{\Omega}(x,y)}{\delta(y)} \ell^{-\Xi}(\delta(y)) ~dy  \lesssim \delta^{\gamma}(x) \int_{\Omega_5} \frac{1}{|x-y|^{N} \ell^{-\Xi}(\delta(y)) } ~dy \\
	& \lesssim \delta(x) \ell^{-\Xi}(\delta(x)) \int_{\Omega_5} \frac{1}{|x-y|^{N}} ~dy \lesssim \delta(x) \ell^{-\Xi}(\delta(x)) \int_{\delta(x)/2}^{3 \delta(x)/2}  \int_{\delta(x)/2}^1 \frac{t^{N-2}}{((\delta(x)-h)+t)^{N}} ~dt~dh\\
	& \lesssim \delta(x) \ell^{-\Xi}(\delta(x)) \int_{\delta(x)/2}^1 t^{-1} \int_{-\delta(x)/2t}^{\delta(x)/2t}   \frac{1}{(|r|+1)^{N}} ~dr~dt\\
	& \lesssim \delta(x) \ell^{-\Xi}(\delta(x)) \int_{1/2}^{1/\delta(x)} \rho^{-1} \int_{0}^{1/\rho} \frac{1}{(r+1)^{N}} dr d\rho \leq C(\eta) \ell^{1-\Xi}(\delta(x)) \delta(x).
	\end{split}
	\end{equation*}
 \vspace{0.1cm}\\
\textbf{Case 2:} $x \in \Omega \setminus \Omega_{\frac{\eta}{2}}.$ \vspace{0.1cm}\\ Using the estimates from the proof of Lemma 3.2  \cite[p.37]{AbaGomVaz_2019}, for $\eta$ small enough, we obtain 
	\begin{equation*}
	\begin{split}
	I_1(x) 	&\leq \ell^{-\Xi}(\eta) \delta(x) \left( \int_{\Omega_{\eta/4}}\frac{G^{\Omega}(x,y)}{\delta(y)}~dy + \int_{\Omega_{\eta} \setminus \Omega_{\eta/4}} \frac{G^{\Omega}(x,y)}{\delta(y)}~dy \right) \\
	& \leq C \ell^{-\Xi}(\eta) \delta(x) \left( 1 + \frac{1}{\eta}\right)  \leq C(\mathcal{D}_{\Omega},\eta) \ell^{1-\Xi}(\delta(x)) \delta(x).
	\end{split}
	\end{equation*}
For any $x \in \Omega$, we have the following estimate for $I_2(x):$
	\begin{equation}\label{I2:est}
	\begin{split}
	I_2(x) &\leq  \frac{\ln^{-\Xi}\left(2\right)}{\eta} \mathbb{G}^\Omega[\chi_\Omega](x) \leq  C(\mathcal{D}_{\Omega}) \ell^{1-\Xi}(\delta(x)) \delta(x) \frac{1}{\eta \ell^{1-\Xi}\left(\frac{\eta}{2}\right)} \\
	& \leq C(\mathcal{D}_{\Omega},\eta) \ell^{1-\Xi}(\delta(x)) \delta(x).
	\end{split}
	\end{equation}
Finally, by collecting all the estimates in $\{\Omega_i\}$ in \textbf{Case 1}, and \textbf{Case 2} and \eqref{I2:est}, we get the desired upper estimate.
\end{proof}
Denote 
\begin{equation} \label{alphabeta}\beta:= \frac{2\gamma+ \zeta}{\gamma+1} \quad \kappa:=2-\beta = \frac{2-\zeta}{\gamma+1} 
\end{equation}
\begin{Lem}\label{lem:greenest}
For $0 <\epsilon, \eta < 1$, $N \geq 3$ and $\gamma > 1$, there exist positive constant $C_1, C_2 >0$ such that the following hold: 
	\begin{equation}\label{est:green:lower:1}
	\mathbb{G}^{\Omega}\left[\frac{1}{(\delta+\epsilon^{1/\kappa})^\beta} \chi_{\Omega_{\eta}}\right](x) \gtrsim  \left(\frac{1}{2} (\delta(x)+ \epsilon^{1/\kappa})^\kappa - \epsilon\right) \quad  \forall \ x \in \Omega_{\frac{\eta}{2}}
	\end{equation}
	and 
	\begin{equation}\label{est:green:lower:2}
	\mathbb{G}^{\Omega}\left[\frac{1}{(\delta+\epsilon^{1/\kappa})^\beta} \right](x) \gtrsim  (\delta(x)+ \epsilon^{1/\kappa})^\kappa - \epsilon \quad  \forall \ x \in \Omega_{\frac{\eta}{2}}^c.
	\end{equation}
\end{Lem}
\begin{proof}
	Denote $\epsilon_1= \epsilon^{1/\kappa}$ and $\mathcal{B}_\delta^x:= \{y \in \Omega: |x-y|< \frac{\delta(x)}{2}\} \subset \Omega_\eta.$ Fix $x \in \Omega_{\frac{\eta}{2}}$. Then, for $y \in \mathcal{B}_\delta^x$, we have 	\begin{equation}\label{est:lem:lower}
	\left( \frac{\delta(x)}{|x-y|} \wedge 1 \right) \geq 1, \left( \frac{\delta(y)}{|x-y|} \wedge 1 \right) \geq 1 \quad \text{and }\ \frac{1}{(\delta(y)+ \epsilon_1)^{\beta}} \geq \left(\frac{2}{3}\right)^\beta \frac{1}{(\delta(x)+ \epsilon_1)^{\beta}}.
	\end{equation}
	Using \eqref{est:lem:lower}, we obtain the following
	\begin{equation}\label{est:lem:1}
	\begin{split}
	\mathbb{G}^{\Omega}&\left[\frac{1}{(\delta+\epsilon_1)^\beta} \chi_{\Omega_{\eta}}\right](x) = \int_{\Omega_{\eta}} \frac{\mathcal{G}^{\Omega}(x,y)}{(\delta(y)+ \epsilon_1)^\beta} ~dy \\
	&\geq C \int_{\mathcal{B}_\delta^x} \frac{1}{(\delta(y)+ \epsilon_1)^\beta} \frac{1}{|x-y|^{N-2}}  \left( \frac{\delta(x)\delta(y)}{|x-y|^2} \wedge 1 \right) \,dy,\\
	& \geq C \int_{\mathcal{B}_\delta^x} \frac{1}{(\delta(y)+ \epsilon_1)^\beta} \frac{1}{|x-y|^{N-2}}  \left( \frac{\delta(x)}{|x-y|} \wedge 1 \right) \left( \frac{\delta(y)}{|x-y|} \wedge 1 \right) \,dy,\\
	& \geq \frac{C'}{(\delta(x)+ \epsilon_1)^\beta} \int_{\mathcal{B}_\delta^x}  \frac{1}{|x-y|^{N-2}} ~dy = \frac{C'' \delta(x)^{2}}{(\delta(x)+ \epsilon_1)^\beta} \geq C'' \left(\frac{1}{2} (\delta(x)+ \epsilon_1)^\kappa - \frac{\epsilon_1^2}{(\delta(x)+ \epsilon_1)^\beta}\right)\\
	&\geq C_1 \left( \frac{1}{2}  (\delta(x)+\epsilon^{1/\kappa})^{\kappa} -\epsilon\right) . 
	\end{split}
	\end{equation}
	For the second claim: let $x \in \Omega_{\frac{\eta}{2}}^c$. Then by using \cite[Theorem 3.4]{AbaGomVaz_2019}, we get
	\begin{equation*}
	\begin{split}
	\mathbb{G}^{\Omega}\left[\frac{1}{(\delta+\epsilon_1)^\beta} \right](x) & \geq \frac{1}{(\mathcal{D}_{\Omega}+ 1)^\beta}  \mathbb{G}^{\Omega}[\chi_{\Omega}](x) \geq c\delta(x)  \geq c_2 (\delta(x)+ \epsilon^{1/\kappa})^\kappa - \epsilon,
	\end{split}
	\end{equation*}
	where $c_2=c_2(\eta, \kappa,\mathcal{D}_{\Omega})$.
\end{proof}
\subsection{Singular weights}
\subsubsection{Boundary behavior}
Let $\varphi_1$ and $\lambda_1$ are first eigenfunction and eigenvalue for the mixed local-nonlocal operator $(-\Delta) + (-\Delta)^s$ such that (\cite[Proposition 3.7]{ChaGomVaz_2020} and \cite[Proposition 5.3, 5.4]{BonFigVaz_2018}) 
\begin{equation}\label{est:eigen}
\varphi_1 \asymp \delta^\gamma \ \text{and} \ \varphi_1= \lambda_1 \mathbb{G}^\Omega[\varphi_1] \quad \text{in} \ \Omega
\end{equation}
\begin{Lem}\label{bdry:est}
Let $u_n$ be the weak solution of the problem $(P_n)$, then for any $\gamma>0$ and $\zeta \geq 0$, there exist a constant $C_0>0$ independent of $n$ and $x\in \Omega$ such that $C_0 \delta(x) \leq u_n(x)$. Moreover, for $\zeta \in [0,2)$, we have
\begin{enumerate}
    \item if $\zeta+ \gamma \leq 1$, there exist two constants $C_1, C_2>0$ independent of $n$ and $x \in \Omega$ such that
    \begin{equation}\label{est:upper-lower:appr:solu:weak}
	C_1 \delta(x) \leq u_n(x) \leq C_2 \delta(x) \left\{
	\begin{array}{ll}
	1  & \text{ if } \  \zeta + \gamma < 1, \vspace{0.1cm}\\
	\ln\left(\frac{\mathcal{D}_\Omega}{\delta(x)}\right) & \text{ if } \ \zeta + \gamma = 1,
	\end{array} 
	\right.\ \text{for} \ x \in \Omega.
	\end{equation}
	 \item if $\zeta+ \gamma > 1$, there exist two constants $C^1, C^2$ and $C^3>0$ independent of $n$ and $x \in \Omega$ such that
    \begin{equation}\label{est:upper-lower:appr:solu:strong}
	C^1 (\delta(x) + n^{\frac{-(1+\gamma)}{(2-\zeta)}})^{\frac{2-\zeta}{\gamma+1}} - C^2 n^{-1} \leq u_n(x) \leq C^3 \delta^{\frac{2-\zeta}{\gamma+1}}(x) \ \text{for} \ x \in \Omega.
	\end{equation}
    \end{enumerate}
\end{Lem}
\begin{proof}
To prove the boundary behavior, we divide our study into three cases: \vspace{0.1cm}\\
\textbf{Case 1:} $\zeta+ \gamma \leq 1$ \vspace{0.1cm}\\
 Using the integral representation in \eqref{int:rep}, lower boundary behavior in \eqref{est:appr:solu:small:lower}, $f_n \leq f$, and Lemma \ref{est:green}, we get
	\begin{equation}\label{est:appr:solu:small:upper}
	\begin{split}
	u_n(x) & =\int_{\Omega} \frac{G(x,y)  f_n(y)}{(u_n + \frac{1}{n})^\gamma} ~dy \leq \int_{\Omega} \frac{G(x,y)  f(y)}{u_n^\gamma} ~dy  \\
	&\lesssim  \mathbb{G}^{\Omega}\left[\frac{1}{\delta^{\gamma+\zeta}}\right] \lesssim \delta(x) \left\{
	\begin{array}{ll}
	1  & \text{ if } \  \zeta + \gamma < 1, \vspace{0.1cm}\\
	\ln\left(\frac{\mathcal{D}_\Omega}{\delta(x)}\right) & \text{ if } \ \zeta + \gamma = 1,
	\end{array} 
	\right. 
	\end{split}
	\end{equation}
	Using \eqref{est:appr:solu:small:upper} and estimates in Theorem \ref{Green:thm:1}, we obtain our first claim. \vspace{0.1cm}\\
	\textbf{Case 2:} $\zeta+ \gamma > 1$. \vspace{0.1cm}\\
For $n \in \mathbb{N}$ and $\kappa >0$, we define $h_{\epsilon_n}:= \frac{1}{(\delta+\epsilon_n^{1/\kappa})^{2-\kappa}} \in L^\infty(\Omega)$ with $\epsilon_n=\frac{1}{n}.$ Since $\gamma>1$ and $\zeta<2$, we choose $$\kappa=\frac{2-\zeta}{\gamma+1} \quad \text{such that}\quad \kappa \gamma + \zeta = 2-\kappa.$$
Then, from Lemma \ref{lem:prelimi}, there exists a unique positive weak solution $v_{\epsilon_n} \in H_0^1(\Omega) \cap L^\infty(\Omega)$ of the following problem
	\begin{equation}\tag{$S_{\epsilon_n}$} \left\{
\begin{aligned}
-\Delta v_{\epsilon_n} + (-\Delta)^s v_{\epsilon_n} & =h_{\epsilon_n}, \quad v_{\epsilon_n}>0 \quad &&\text{ in }  \Omega, \\
v_{\epsilon_n} & = 0 &&  \text{ in } \ \mathbb{R}^N \setminus \Omega,
\end{aligned}
\right.
\end{equation}
Set $\epsilon_1^n= {\epsilon_n}^{1/\kappa}>0$. Using the uniqueness property of $v_{\epsilon_n}$, integral representation of the solution via Green function and Lemma \ref{est:green}, we obtain
	\begin{equation}\label{est:weight:approx1}
	\begin{split}
	v_{\epsilon_n}= \mathbb{G}^\Omega\left[\frac{1}{(\delta+\epsilon_1^n)^{2-\kappa}}\right] \leq \mathbb{G}^\Omega\left[\frac{1}{\delta^{2-\kappa}}\right] \leq C \delta^{\kappa} \quad \text{in }  \Omega \quad \text{if} \ \kappa >1,
	\end{split}
	\end{equation}
	and on the other hand for $0< \eta <1$, Lemma \ref{lem:greenest} gives
	\begin{equation}\label{est:weight:approx2}
	\begin{split}
	v_{\epsilon_n}&(x)= \mathbb{G}^\Omega\left[\frac{1}{(\delta+\epsilon_1^n)^{2-\kappa}}\right](x) \\
	& 
	\begin{aligned}
	\gtrsim \left\{
	\begin{array}{ll}
	\mathbb{G}^{\Omega}\left[\frac{1}{(\delta+\epsilon_1^n)^{2-\kappa}} \chi_{\Omega_{\eta}}\right](x) \gtrsim  \left(\frac{1}{2}(\delta(x)+ \epsilon_1^n)^\kappa - \epsilon_n \right) &  \text{ if } x \in \Omega_{\frac{\eta}{2}},\\
	 (\delta(x)+ \epsilon_1^n)^\kappa - \epsilon_n &  \text{ if } x \in \Omega \setminus \Omega_{\frac{\eta}{2}}, \\
	\end{array} 
	\right.
	\end{aligned}\\
	& \gtrsim \left( \frac{1}{2}(\delta(x)+ \epsilon_1^n)^\kappa - \epsilon_n\right),   \quad x \in \Omega \quad \text{if} \ \gamma >1 \ \text{and} \ \zeta<2.\\
	\end{split}
	\end{equation}
Collecting the estimates in \eqref{est:weight:approx1}-\eqref{est:weight:approx2}, there exists a constant $C_3, C_4>0$ independent of $n$ such that
	\begin{equation}\label{est:weight:approx3}
	C_3 \left(\frac{1}{2}(\delta(x)+ \epsilon_1^n)^\kappa - \epsilon_n\right)   \leq v_{\epsilon_n}(x) \leq C_4 \delta(x)^\kappa, \ x \in \Omega \quad \text{if} \ \gamma >1 \ \text{and} \ \zeta<2.   
	\end{equation}
	
\noindent Define 
	\begin{equation*}
	\underline u^{\epsilon_n }= c_\eta v_{\epsilon_n} \ \text{with} \ 0< c_\eta < \frac{C_1}{C_4} \left(\frac{\eta}{2}\right)^{1-\kappa},
	\end{equation*}
	where $C_1, C_4$ are defined in \eqref{est:upper-lower:appr:solu:weak} and \eqref{est:weight:approx3} respectively. We note that $c_\eta$ is independent of $\epsilon_n$ such that
	$$ \underline u^{\epsilon_n } \leq u_n \ \text{in } \Omega \setminus \Omega_{\frac{\eta}{2}} 
	$$
	and
	$$ 
	\left(\underline u^{\epsilon_n } + \epsilon_n\right) \leq 2 C_4 c_\eta (\delta + \epsilon_1^n)^\kappa + (1-C_4 c_\eta) \epsilon_n \quad \text{ in } \Omega.
	$$
	
\noindent If $2 C_4 c_\eta (\delta(x)+ \epsilon_1^n)^\kappa \geq (1-C_4 c_\eta) \epsilon_n$ then by choosing $\eta$ small enough such that $c_\eta < \mathcal{G}_1^{\frac{1}{\gamma+1}} (4C_4)^{-\frac{\gamma}{\gamma+1}}$, we have
\[
\begin{split}
    (-\Delta) \underline u^{\epsilon_n } + (-\Delta)^s \underline u^{\epsilon_n } = c_\eta  (\delta+ \epsilon_1^n)^{-(2-\kappa)}  < \mathcal{G}_1(4C_4 c_\eta)^{-\gamma} (\delta+ \epsilon_1^n)^{-\kappa \gamma-\zeta} \leq f_{\epsilon_n}(x) \left(\underline u^{\epsilon_n } + \epsilon_n\right)^{-\gamma} \quad  \text{in} \ \Omega.
\end{split}
\]
where $\mathcal{G}_1$ is defined in \eqref{weight:approx}. \vspace{0.1cm}\\
If $2 C_4 c_\eta (\delta(x)+ \epsilon_1^n)^\kappa \leq (1-C_4 c_\eta) \epsilon_n$ then again by choosing $\eta$ small enough such that $C_4 c_\eta <1$ and $c_\eta < \mathcal{G}_1 (2(1-C_4 c_\eta))^{-q} $ , we have 
\[
\begin{split}
  (-\Delta) \underline u^{\epsilon_n } + (-\Delta)^s \underline u^{\epsilon_n }& =c_\eta  (\delta+ \epsilon_1^n)^{-(2-\kappa)} \leq \mathcal{G}_1 (2(1-C_4 c_\eta))^{-q}  (\delta+ \epsilon_1^n)^{-\kappa \gamma -\zeta}\\
  &\leq f_{\epsilon_n}(x)(1-C_4 c_\eta)^{-q} (2{\epsilon_n})^{-\gamma} \leq f_{\epsilon_n}(x)\left(\underline u^{\epsilon_n } + \epsilon_n\right)^{-\gamma} \quad  \text{in} \ \Omega.
\end{split}
\]
Now, by considering both the cases and applying the weak comparison principle in $\Omega_{\frac{\eta}{2}}$ for $u_{n}$ and $ \underline u^{\epsilon_n}$, we get $\underline u^{\epsilon_n} \leq u_n$ in $\Omega$, namely there exist constants $0 < C^1, C^2 < \frac{1}{2}$ (by taking $\eta$ small enough) such that
	\begin{equation}\label{est:appr:solu:large:lower}
	C^1 (\delta + \epsilon_1^n)^\kappa - C^2 {\epsilon_n} \leq u_n \quad \text{in } \Omega.
	\end{equation}
Finally, by using the integral representation, we obtain
	\begin{equation}\label{est:appr:solu:large:upper}
	\begin{split}
	u_n & = \mathbb{G}^{\Omega}\left[\frac{f_{\epsilon_n}(x)}{(u_n + \epsilon_n)^\gamma}\right] \leq \mathbb{G}^{\Omega}\left[\frac{f_{\epsilon_n}(x)}{(C^1 (\delta+ \epsilon_1^n)^\kappa +  (1- C^2) {\epsilon_n})^\gamma}\right] \\
	& \lesssim \mathbb{G}^{\Omega}\left[\frac{1}{\delta^{\kappa \gamma+ \zeta}}\right] =\mathbb{G}^{\Omega}\left[\frac{1}{\delta^{2-\kappa }}\right] \lesssim  \delta^\kappa \quad \text{if} \ \gamma >1 \ \text{and} \ \zeta <2. 
	\end{split}
	\end{equation}
\end{proof}
\subsubsection{Sobolev regularity estimates}
\begin{Lem}\label{sobolev:reg}
Let $\gamma>0$, $\zeta \in [0,2)$ and $u_n$ be the weak solution of the problem $(P_n)$. Then,
    \begin{equation*}
	u_n^{\frac{\mathfrak{L}+1}{2}} \ \text{is uniformly bounded in} \  H_0^1(\Omega) \ \text{for any} \ \mathfrak{L} > \left\{
	\begin{array}{ll}
	0  & \text{ if } \  \zeta+ \gamma \leq  1, \vspace{0.1cm}\\
	\mathfrak{L}^* & \text{ if } \ \zeta+ \gamma > 1,
	\end{array} 
	\right.
	\ \text{where} \ \mathfrak{L}^*:= \frac{\gamma + \zeta-1}{2-\zeta}.
	\end{equation*}
\end{Lem}
\begin{proof}
Let $n \in \mathbb{N}$ and $u_n$ be the weak solution of the problem $(P_n)$ given by Lemma \ref{Lem:apriori}. Since, $u_n \in L^\infty(\Omega) \cap H_0^1(\Omega)$ and positive, then for any $\epsilon>0$ and $\mathfrak{L} >0$, $(u_n + \epsilon)^{\mathfrak{L}} - \epsilon^{\mathfrak{L}}$ belongs to $H_0^1(\Omega)$, therefore, an admissible test function in \eqref{def:notion:approx}. Taking it so for $\epsilon \in (0, \frac{1}{n})$ and passing $\epsilon \to 0$ as in the proof of Lemma \ref{lem:apri:weak}(i), we obtain
\begin{equation}\label{est:apri:n-1}
 \begin{split}
    \frac{4 \mathfrak{L}}{(\mathfrak{L} +1)^2}\int_{\Omega} |\nabla u_n^{\frac{\mathfrak{L} +1}{2}}|^2 ~dx &+ \frac{2 C(N,s) \mathfrak{L}}{(\mathfrak{L} +1)^2} \int_{\mathbb{R}^N} \int_{\mathbb{R}^N} \frac{((u_n^\frac{\mathfrak{L} +1}{2}(x)- u_n^\frac{\mathfrak{L} +1}{2}(y))^2}{|x-y|^{N+2s}} ~dx ~dy \\
    &\leq \int_{\Omega} f_n(x) u_n^{\mathfrak{L} - \gamma} ~dx:= \mathfrak{G}(u_n).
\end{split}   
\end{equation}
Now to estimate the right hand side term in \eqref{est:apri:n-1}, we divide the proof into three cases. Using \eqref{weight:approx} and \eqref{est:upper-lower:appr:solu:weak}, we obtain the following: \vspace{0.1cm}\\
\textbf{Case 1:} $\zeta+ \gamma < 1$ 
\begin{equation}\label{est:apri:n-2}
  \begin{split}
    \mathfrak{G}(u_n) &\leq \int_{\Omega} f_n(x) u_n^{\mathfrak{L} - \gamma} ~dx \leq \mathcal{G}_2 \int_{\Omega} \frac{u_n^{\mathfrak{L} - \gamma}}{\delta^\zeta(x)} ~dx \lesssim \int_{\Omega} u_n^{\mathfrak{L} - \gamma -\zeta} ~dx \\
    & \lesssim \int_{\Omega} \delta^{\mathfrak{L} - \gamma -\zeta}(x) ~dx \leq C \quad \text{if} \quad \mathfrak{L} > \gamma + \zeta-1.
\end{split}   
\end{equation}
where $\mathfrak{L} > \gamma + \zeta-1$ holds trivially. 
\vspace{0.1cm}\\
\textbf{Case 2:} $\zeta+ \gamma = 1$ \vspace{0.1cm}\\Since $\mathfrak{L} >0$, we can choose $\chi \in (0, \mathfrak{L})$ small enough such that 
\[
\max\left\{\ln^\zeta\left(\frac{\mathcal{D}_\Omega}{\delta(x)}\right), \ln^{\mathfrak{L}-\gamma}\left(\frac{\mathcal{D}_\Omega}{\delta(x)}\right)\right\} \leq C(\mathcal{D}_\Omega, \zeta, \mathfrak{L}, \gamma) \delta^{-\chi}(x) \quad \text{for all} \ x \in \Omega.
\]
Then, we have
\begin{equation}\label{est:apri:n-3}
  \begin{split}
    \mathfrak{G}(u_n) &\leq \int_{\Omega} f_n(x) u_n^{\mathfrak{L} - \gamma} ~dx \leq \mathcal{G}_2 \int_{\Omega}  \frac{u_n^{\mathfrak{L} - \gamma}}{\delta^\zeta(x)} ~dx \lesssim \int_{\Omega} \ln^\zeta\left(\frac{\mathcal{D}_\Omega}{\delta(x)}\right) u_n^{\mathfrak{L} - 1} ~dx \\
    & \lesssim \left\{
	\begin{array}{ll}
	\int_{\Omega} \ln^{\mathfrak{L}-\gamma}\left(\frac{\mathcal{D}_\Omega}{\delta(x)}\right) \delta^{\mathfrak{L} - 1}(x) ~dx  & \text{ if } \  \mathfrak{L}  \geq 1, \vspace{0.1cm}\\
	\int_{\Omega} \ln^\zeta\left(\frac{\mathcal{D}_\Omega}{\delta(x)}\right) \delta^{\mathfrak{L} - 1}(x) ~dx & \text{ if } \ \mathfrak{L} < 1,
	\end{array} 
	\right.\\
	& \lesssim \int_{\Omega} \delta^{\mathfrak{L} - 1-\chi}(x) ~dx \leq C \quad \text{if} \ \mathfrak{L} > \chi.
\end{split}   
\end{equation}
\textbf{Case 3:} $\zeta+ \gamma > 1$ 
\begin{equation}\label{est:apri:n-4}
  \begin{split}
    \mathfrak{G}(u_n) &\leq \int_{\Omega} f_n(x) u_n^{\mathfrak{L} - \gamma} ~dx \leq \mathcal{G}_2 \int_{\Omega} \frac{u_n^{\mathfrak{L} - \gamma}}{\delta^\zeta(x)} ~dx = \mathcal{G}_2 \int_{\Omega} u_n^{\mathfrak{L} - \gamma} \left(\delta^\frac{2-\zeta}{(\gamma+1)}(x)\right)^\frac{-\zeta(\gamma+1)}{2-\zeta} ~dx\\
    &\lesssim \int_{\Omega} u_n^{\mathfrak{L} - \gamma -\frac{\zeta(\gamma+1)}{2-\zeta}} ~dx \lesssim \int_{\Omega} \delta^{\left(\mathfrak{L} - \gamma\right) \frac{2-\zeta}{(\gamma+1)} -\zeta}(x) ~dx \leq C \quad \text{if} \quad \mathfrak{L} > \frac{\gamma + \zeta-1}{2-\zeta}.
\end{split}   
\end{equation}
Collecting the estimates in \eqref{est:apri:n-2}-\eqref{est:apri:n-4}, we obtain:
\[
     \int_{\Omega} |\nabla u_n^{\frac{\mathfrak{L} +1}{2}}|^2 ~dx \leq  \frac{(\mathfrak{L} +1)^2}{4 \mathfrak{L}}\mathfrak{G}(u_n) \leq C(\mathfrak{L}, \gamma, \zeta) \quad \text{for all} \ \mathfrak{L} > \left\{
	\begin{array}{ll}
	0  & \text{ if } \  \zeta+ \gamma \leq  1, \vspace{0.1cm}\\
	\frac{\gamma + \zeta-1}{2-\zeta} & \text{ if } \ \zeta+ \gamma > 1.
	\end{array} 
	\right.
\]
\end{proof}
\section{Existence, summability and Sobolev regularity results}\label{sec:proof}
\subsection{Proof of Theorem \ref{result:1}}
Let $(r, \gamma) \in \mathcal{P}_{r,\gamma} \cap \{(r, \gamma): r \in [1, r^\sharp), 0 \leq \gamma <1\}$ and $u_n$ be the weak solution of the problem $(P_n)$ in the sense that
\begin{equation}\label{sense:approx}
\begin{split}
    \int_{\Omega} \nabla u_n \cdot \nabla \psi ~dx + \frac{C(N,s)}{2} \int_{\mathbb{R}^N} \int_{\mathbb{R}^N} &\frac{(u_n(x)- u_n(y))(\psi(x)- \psi(y))}{|x-y|^{N+2s}} ~dx ~dy \\
    &= \int_{\Omega} \frac{f_n(x)}{(u_n + \frac{1}{n})^\gamma} \psi ~dx \quad \forall \psi \in H_0^1(\Omega).
\end{split}
\end{equation}
From Lemma \ref{lem:apri:weak}(iv), we know that $u_n$ is uniformly bounded in $W_0^{1,q}(\Omega)$ with $q:= \frac{Nr(1+\gamma)}{N-r(1-\gamma)}.$ Then, there exists a $u \in W_0^{1, q}(\Omega)$ such that
\begin{equation}\label{convergence:est:case:1}
\begin{split}
    & u_n \rightharpoonup u \ \text{in} \ W_0^{1,q}(\Omega), \ \ u_n \to u \ \text{in} \ L^{j}(\Omega) \ \ \text{for} \ 1 \leq j < \sigma_r\ \text{and a.e. in} \ \mathbb{R}^N.
\end{split}
\end{equation}
Since, for any $\psi \in W_0^{1,q'}(\Omega)$, \[W_0^{1,q}(\Omega) \ni f \mapsto \int_{\Omega} \nabla f \cdot \nabla \psi ~dx + \frac{C(N,s)}{2} \int_{\mathbb{R}^N} \int_{\mathbb{R}^N} \frac{(f(x)- f(y))(\psi(x)- \psi(y))}{|x-y|^{N+2s}} ~dx ~dy\]
is a bounded linear functional on $W_0^{1,q}(\Omega).$ Then for every $\psi \in W_0^{1, q'}(\Omega)$, we obtain
\begin{equation}\label{exis:est1}
\begin{split}
  \int_{\Omega} \nabla u_n &\cdot \nabla \psi ~dx + \frac{C(N,s)}{2} \int_{\mathbb{R}^N} \int_{\mathbb{R}^N} \frac{(u_n(x)- u_n(y))(\psi(x)- \psi(y))}{|x-y|^{N+2s}} ~dx ~dy \\
& \to \int_{\Omega} \nabla u \cdot \nabla \psi ~dx + \frac{C(N,s)}{2} \int_{\mathbb{R}^N} \int_{\mathbb{R}^N} \frac{(u(x)- u(y))(\psi(x)- \psi(y))}{|x-y|^{N+2s}} ~dx ~dy \quad \text{as} \ n \to \infty
\end{split}  
\end{equation}
and using the fact that $f_n \leq f$, $C_1(\omega) \leq u_1 \leq u_n$ a.e. in $\omega \Subset \Omega$ and Lebesgue dominated convergence theorem we get
\begin{equation}\label{exis:est2}
   \begin{split}
   \int_{\Omega} \frac{f_n(x)}{(u_n + \frac{1}{n})^\gamma} \psi ~dx \to \int_{\Omega} \frac{f(x)}{u^\gamma} \psi ~dx  \quad \text{as} \ n \to \infty \quad \text{for $\psi \in L^{r'}(\Omega)$ with $\supp(\psi) \Subset \Omega$}.
\end{split} 
\end{equation}
Passing limit $n \to \infty$ in \eqref{sense:approx} and using \eqref{exis:est1}-\eqref{exis:est2}, we obtain our claim.
\qed 

\subsection{Proof of Theorem \ref{result:2}}
Let $(r, \gamma) \in \mathcal{P}_{r,\gamma} \setminus \{(r, \gamma): r \in [1, r^\sharp), 0 \leq \gamma <1\}$ and $u_n$ be the weak solution of the problem $(P_n)$ and satisfies \eqref{sense:approx}. From Lemma \ref{lem:apri:weak}, we know that $u_n^{\frac{\mathfrak{S}_r+1}{r}}$ is uniformly bounded in $H_0^{1}(\Omega).$ The condition $(r, \gamma) \in \mathcal{P}_{r,\gamma} \setminus \{(r, \gamma): r \in [1, r^\sharp), 0 \leq \gamma <1\}$ implies $\mathfrak{S}_r \geq 1$. Together with the fact that for every compact subset $\omega \Subset \Omega$ there exists $C=C(\omega)$ independent of $n$ such that $0< C(\omega) \leq u_n(x)$ for $x \in \omega$, we get $u_n$ is uniformly bounded in $H_{loc}^1(\Omega)$. Precisely,
\begin{equation*}
\begin{split}
    \int_{\omega} |\nabla u_n|^2 ~dx \leq C^{-(\mathfrak{S}_r-1)} \int_{\omega} u_n^{(\mathfrak{S}_r-1)}  |\nabla u_n|^2 ~dx \leq \frac{4 C^{-(\mathfrak{S}_r-1)}}{(\mathfrak{S}_r +1)} \int_{\omega} |\nabla u_n^{\frac{\mathfrak{S}_r+1}{2}}|^2 ~dx \leq C_1
\end{split}
\end{equation*}
where $C_1$ is independent of $n.$ Then, there exists a $u \in H_{loc}^1(\Omega)$ such that
\begin{equation}\label{convergence:est:case:2}
\begin{split}
    & u_n \rightharpoonup u \ \text{in} \ H_{loc}^1(\Omega), \ \ u_n \to u \ \text{in} \ L^{j}_{loc}(\Omega) \ \text{for} \ 1 \leq j < 2^\ast\ \text{and a.e. in} \ \mathbb{R}^N.
\end{split}
\end{equation}
Therefore, by using the weak convergence property and adopting the same arguments from \cite[Theorem 3.6]{Canino_et_al}, we are able to pass limits in the left hand side of \eqref{sense:approx}, {\it i.e.} for any $\psi \in H^1_{loc}(\Omega)$ with $\supp(\psi) \Subset \Omega$, 
\begin{equation}\label{exis:est3}
\begin{split}
  \int_{\Omega} \nabla u_n &\cdot \nabla \psi ~dx + \frac{C(N,s)}{2} \int_{\mathbb{R}^N} \int_{\mathbb{R}^N} \frac{(u_n(x)- u_n(y))(\psi(x)- \psi(y))}{|x-y|^{N+2s}} ~dx ~dy \\
& \to \int_{\Omega} \nabla u \cdot \nabla \psi ~dx + \frac{C(N,s)}{2} \int_{\mathbb{R}^N} \int_{\mathbb{R}^N} \frac{(u(x)- u(y))(\psi(x)- \psi(y))}{|x-y|^{N+2s}} ~dx ~dy \quad \text{as} \ n \to \infty.
\end{split}  
\end{equation}
Now, by repeating the same arguments as from the proof of Theorem \ref{result:1} for the limit passage in the right hand side term and by passing limits $n \to \infty$ using \eqref{exis:est3}, we obtain our claim.
\qed 

\subsection{Proof of Theorem \ref{result:3}}
Let $u$ be the weak solution of the problem \eqref{main:problem} obtained from the approximable solution $u_n$ of the problem $(P_n)$.
\begin{enumerate}
    \item[(i)] Let $r \in [1, \frac{N}{2})$ and $N > 2$. Then from Lemma \ref{lem:apri:weak}, we know that $u_n^{\frac{\mathfrak{S}_r+1}{2}}$ is uniformly bounded in $H_0^1(\Omega)$ which further implies that there exists a $v \in H_0^1(\Omega)$ such that $u_n^{\frac{\mathfrak{S}_r+1}{2}} \rightharpoonup v \ \text{in} \ H_0^1(\Omega)$ and by Sobolev's embedding, we get $u_n^{\frac{\mathfrak{S}_r+1}{2}} \to v \ \text{in} \ L^j(\Omega)$ for every $1\leq j < 2^\ast$ and a.e. in $\mathbb{R}^N.$ Together with convergence estimate in \eqref{convergence:est:case:1} and \eqref{convergence:est:case:2}, we obtain $v=u^{\frac{\mathfrak{S}_r+1}{2}}$  a.e. in $\mathbb{R}^N.$ The remaining estimates follows from using the embedding $L^r(\Omega) \hookrightarrow  L^j(\Omega)$ and repeating the same proof by replacing $\mathfrak{S}_r$ by $\mathfrak{S}_j.$
    \item[(ii)] Let $r=\frac{N}{2}$ and $N \geq 2.$ Again from Lemma \ref{lem:apri:weak}, we imply that $\mathfrak{H} \left(\frac{u_n}{2}\right)$ in case of $0 \leq \gamma \leq 1$ and $\mathfrak{D} \left(\frac{u_n}{2}\right)$ in case of $\gamma > 1$ are uniformly bounded in $H_0^1(\Omega)$ which further implies that there exist $v_1, v_2 \in H_0^1(\Omega)$ such that  
    \[
    \begin{split}
        & 0 \leq \gamma \leq 1: \ \mathfrak{H} \left(\frac{u_n}{2}\right) \rightharpoonup v_1 \ \text{in} \ H_0^1(\Omega) \ \text{and} \ \mathfrak{H} \left(\frac{u_n}{2}\right) \to v_1 \  \text{in} \ L^j(\Omega) \ \text{for every} \ 1\leq j < 2^\ast \ \text{and a.e. in} \ \mathbb{R}^N,\\
        &\gamma > 1: \ \mathfrak{D} \left(\frac{u_n}{2}\right) \rightharpoonup v_2\  \text{in} \ H_0^1(\Omega)  \ \text{and} \ \mathfrak{D} \left(\frac{u_n}{2}\right) \to v_2\  \text{in} \ L^j(\Omega) \ \text{ for every} \ 1\leq j < 2^\ast \ \text{and a.e. in} \ \mathbb{R}^N.
    \end{split}\]
Using the above estimate with the convergence properties in \eqref{convergence:est:case:1} and \eqref{convergence:est:case:2}, we identify the limit functions $v_1$ and $v_2$ as
$$v_1=\mathfrak{H} \left(\frac{u}{2}\right) \quad \text{and} \quad v_2= \mathfrak{D} \left(\frac{u}{2}\right) \quad \text{a.e. in} \ \mathbb{R}^N.$$
\item[(iii)] From  Lemma \ref{lem:apri:weak}, we know that $u_n$ is uniformly bounded in $L^\infty(\Omega)$ when $r> \frac{N}{2}$ and hence $u \in L^\infty(\Omega).$
\item[(iv)] The first part in the claim follows from the proof of Theorem \ref{result:1} and the second part of claim follows using the embedding $L^r(\Omega) \hookrightarrow L^1(\Omega)$ for $1 \leq r \leq \infty$ and by repeating the proof of Theorem \ref{result:1} and \ref{result:2} with $\mathfrak{S}_r=1$.
\end{enumerate}\qed 
\subsection{Proof of Theorem \ref{result:10}} The convergence estimate in the proof of Theorem \ref{result:1} and Theorem \ref{result:2}, and the uniform a priori estimates in Lemma \ref{lem:lower:reg} implies the required claim. The only if statement in proving optimal Sobolev regularity follows from the Hardy inequality and the boundary behavior of the weak solution. Precisely, if \[ \mathfrak{S} \leq 0  \ \text{and} \  \gamma + \frac{1}{r}< 1, 
	 \quad \text{then} \quad u^{\frac{\mathfrak{S}+1}{2}} \notin H_0^{1}(\Omega).\] Indeed, in this case, we have
$$ \|u^\frac{\mathfrak{S}+1}{2}\|_{H_0^1(\Omega)}^2 \geq C\int_{\Omega} \frac{|u^\frac{\mathfrak{S}+1}{2}(x)|^2}{|d(x)|^2} ~dx \geq C\int_{\Omega} d^{\mathfrak{S}-1}(x) ~dx =\infty$$
and hence, we deduce $u^\frac{\mathfrak{S}+1}{2} \notin H_0^{1}(\Omega).$ 
\qed 
\subsection{Proof of Theorem \ref{result:8}} Let $u$ and $v$ be two weak solutions of the problem \eqref{main:problem}-\eqref{bdry:cond} for datum $f$ and $g$ respectively, obtained in Theorem \ref{result:1} and Theorem \ref{result:2}. Let $\{u_n\}_{n \in \mathbb{N}}, \{v_n\}_{n \in \mathbb{N}} \subset H_0^{1}(\Omega) \cap l^\infty(\Omega)$ are sequence of the solutions of the non-singular approximating problem $(P_n)$ for $n \in \mathbb{N}$ where $f_n$ and $g_n$ are increasing sequences such that $f_n \to f$ and $g_n \to g$ in $L^r(\Omega).$ By taking $[(u_n-v_n)_+ + \epsilon]^{\mathfrak{S}_r} -\epsilon^{\mathfrak{S}_r}$ as a test function in \eqref{def:notion:approx}, we obtain
\[
\begin{split}
    &\mathfrak{S}_r \int_{\Omega} [\epsilon+ (u_n-v_n)_+]^{\mathfrak{S}_r-1} \nabla(u_n-v_n) \cdot \nabla(u_n-v_n)_+ ~dx + \int_{\Omega} \frac{g_n}{\left(v_n+ \frac{1}{n}\right)^\gamma} ([(u_n-v_n)_+ + \epsilon]^{\mathfrak{S}_r} -\epsilon^{\mathfrak{S}_r}) ~dx \\
    &+  \frac{C(N,s)}{2} \int_{\mathbb{R}^N} \int_{\mathbb{R}^N} \frac{((u_n-v_n)(x)- (u_n-v_n)(y))([(u_n-v_n)_+ + \epsilon]^{\mathfrak{S}_r}(x)- [(u_n-v_n)_+ + \epsilon]^{\mathfrak{S}_r}(y))}{|x-y|^{N+2s}} ~dx ~dy \\
    &= \int_{\Omega} \frac{f_n}{\left(u_n+ \frac{1}{n}\right)^\gamma} ([(u_n-v_n)_+ + \epsilon]^{\mathfrak{S}_r} -\epsilon^{\mathfrak{S}_r}) ~dx
\end{split}
\]
In order to pass limits in the integrals on the left hand side of the above equality we use Fatou's Lemma and for the integral on the right hand side, we use Lebesgue dominated convergence, since for $\epsilon < \frac{1}{n}$, and using \eqref{choice:1} and \eqref{choice:2}, we have the dominating function
\[
\begin{split}
    \frac{f_n}{\left(u_n+ \frac{1}{n}\right)^\gamma} ([(u_n-v_n)_+ + \epsilon]^{\mathfrak{S}_r} -\epsilon^{\mathfrak{S}_r}) \leq f_n  ( u_n + 1/n)^{\mathfrak{S}_r-\gamma} \leq f (u + 1)^{\mathfrak{S}_r-\gamma} \leq \frac{|f|^r}{r} + \frac{(r-1)}{r} |u+1|^{\sigma_r}.
\end{split}
\]
Therefore, we have
\begin{equation}\label{cont:est0}
    \begin{split}
I_1+ I_2&+ I_3:=   \mathfrak{S}_r \int_{\Omega}  (u_n-v_n)_+^{\mathfrak{S}_r-1} \nabla(u_n-v_n) \cdot \nabla(u_n-v_n)_+ ~dx + \int_{\Omega} \frac{g_n}{\left(v_n+ \frac{1}{n}\right)^\gamma} (u_n-v_n)_+^{\mathfrak{S}_r} ~dx \\
    &+  \frac{C(N,s)}{2} \int_{\mathbb{R}^N} \int_{\mathbb{R}^N} \frac{((u_n-v_n)(x)- (u_n-v_n)(y))((u_n-v_n)_+^{\mathfrak{S}_r}(x)- (u_n-v_n)_+^{\mathfrak{S}_r}(y))}{|x-y|^{N+2s}} ~dx ~dy \\
    & \leq \int_{\Omega} \frac{f_n}{\left(u_n+ \frac{1}{n}\right)^\gamma} (u_n-v_n)_+^{\mathfrak{S}_r} ~dx: =I_4
\end{split}
\end{equation}
Now, we separately estimate the integrals in the above inequality. 
\vspace{0.1cm}\\
\textbf{Estimate for $I_4$:} With the choice of $\mathfrak{S}_r$ in \eqref{choice:1} and \eqref{choice:2},
\begin{equation}\label{cont:est1}
    0 \leq  \frac{f_n}{\left(u_n+ \frac{1}{n}\right)^\gamma} (u_n-v_n)_+^{\mathfrak{S}_r} \leq  f u^{\mathfrak{S}_r-\gamma} \leq \frac{|f|^r}{r} + \frac{(r-1)}{r} |u|^{\sigma_r}, 
\end{equation}
and the fact that $u_n \to u$ and $v_n \to v$ a.e. in $\Omega$, Lebesgue dominated convergence theorem implies
\begin{equation}\label{cont:est4}
   I_4 \to \int_{\Omega} \frac{f}{u^\gamma} (u-v)_+^{\mathfrak{S}_r} ~dx 
\end{equation}
\textbf{Estimate for $I_2 + I_3$:} It is easy to see that
\begin{equation}\label{cont:est2}
    ((u_n-v_n)(x)-(u_n-v_n)(y))((u_n-v_n)_+^{\mathfrak{S}_r}(x)- (u_n-v_n)_+^{\mathfrak{S}_r}(y)) \geq 0 \quad \text{for} \quad (x,y) \in \mathbb{R}^N \times \mathbb{R}^N.
\end{equation}
Now by using the Fatou's lemma and \eqref{cont:est2} we obtain,
\begin{equation}\label{cont:est5}
   \int_{\Omega} \frac{g}{v^\gamma} (u-v)_+^{\mathfrak{S}_r} ~dx \leq \lim\inf_{n \to \infty} I_2 + I_3 
\end{equation}
\textbf{Estimate for $I_1$:} For this we divide the proof into two cases: $\mathfrak{S}_r \leq 1$ and $\mathfrak{S}_r > 1.$ In the first case, for every $\delta>0$, we have
\[
\begin{split}
    \frac{4 \mathfrak{S}_r}{(\mathfrak{S}_r+1)^2} \int_{\Omega} \left|\nabla \left[(u_n-v_n)_+ + \delta\right]^{\frac{\mathfrak{S}_r+1}{2}}\right|^2 ~dx=  \mathfrak{S}_r \int_{\Omega} \frac{\left|\nabla (u_n-v_n)_+\right|^2}{\left((u_n-v_n)_+ + \delta\right)^{1-\mathfrak{S}_r}}~dx \leq I_1.
\end{split}
\]
Together with the estimate in \eqref{cont:est1} and \eqref{cont:est2}, we obtain $\left[(u_n-v_n)_+ + \delta\right]^{\frac{\mathfrak{S}_r+1}{2}}$ is uniformly bounded in $L^2$ and which combined with the a.e. convergence $u_n-v_n \to u-v$ in $H_0^1(\Omega)$, we obtain
\[
\left[(u_n-v_n)_+ + \delta\right]^{\frac{\mathfrak{S}_r+1}{2}} \rightharpoonup \left[(u-v)_+ + \delta\right]^{\frac{\mathfrak{S}_r+1}{2}}  \ \text{in} \ H_0^1(\Omega)
\]
and the weak lower semicontinuity of the norm gives
\[
\begin{split}
    \mathfrak{S}_r \int_{\Omega} \frac{\left|\nabla (u-v)_+\right|^2}{\left((u-v)_+ + \delta\right)^{1-\mathfrak{S}_r}}~dx &= \frac{4 \mathfrak{S}_r}{(\mathfrak{S}_r+1)^2} \int_{\Omega} \left|\nabla \left[(u-v)_+ + \delta\right]^{\frac{\mathfrak{S}_r+1}{2}}\right|^2 ~dx\\
    &\leq \lim\inf_{n \to \infty} \frac{4 \mathfrak{S}_r}{(\mathfrak{S}_r+1)^2} \int_{\Omega} \left|\nabla \left[(u_n-v_n)_+ + \delta\right]^{\frac{\mathfrak{S}_r+1}{2}}\right|^2 ~dx \leq \lim\inf_{n \to \infty} I_1.
\end{split}
\]
Using Monotone convergence theorem by taking $\delta$ decreasing to $0$, we get
\begin{equation}\label{cont:est3}
    \begin{split}
    \mathfrak{S}_r \int_{\Omega} \frac{\left|\nabla (u-v)_+\right|^2}{(u-v)_+^{1-\mathfrak{S}_r}}~dx \leq \lim\inf_{n \to \infty} I_1.
\end{split}
\end{equation}
In the second case, $\mathfrak{S}_r >1$, taking $\delta=0$ and then reasoning as above, we get the same estimate \eqref{cont:est3}. Now, by passing $n \to \infty$ in \eqref{cont:est0} and using the above convergence estimates in \eqref{cont:est4}, \eqref{cont:est5} and \eqref{cont:est3}, we obtain
\begin{equation}\label{cont:est6}
    \begin{split}
        \frac{4 \mathfrak{S}_r}{(\mathfrak{S}_r+1)^2} \int_{\Omega} |\nabla (u-v)_+^{\frac{\mathfrak{S}_r+1}{2}}|^2 ~dx &\leq \int_{\Omega} \left(\frac{f(x)}{u^\gamma} - \frac{g(x)}{v^\gamma} \right) (u-v)_+^{\mathfrak{S}_r} ~dx \leq \int_{\Omega} (f(x)-g(x)) \frac{(u-v)_+^{\mathfrak{S}_r}}{u^\gamma} ~dx \\
        & \leq \int_{\Omega} (f(x)-g(x)) (u-v)_+^{\mathfrak{S}_r-\gamma} ~dx.
    \end{split}
\end{equation}
Now, by using H\"older inequality, Sobolev embeddings and using the relation of $\mathfrak{S}_r$ in \eqref{choice:1} and \eqref{choice:2}, we obtain
\[
\begin{split}
   \frac{4 \mathfrak{S}_r}{(\mathfrak{S}_r+1)^2} & \int_{\Omega} |\nabla (u-v)_+^{\frac{\mathfrak{S}_r+1}{2}}|^2 ~dx \leq \|f-g\|_{L^r(\Omega_1)} \left(\int_{\Omega} \left((u-v)_+^{\frac{(\mathfrak{S}_r +1)}{2}}\right)^{\frac{2N}{N-2}} ~dx\right)^\frac{1}{r'}\\
 & \leq \|f-g\|_{L^r(\Omega_1)} \left( S(N) \int_{\Omega} |\nabla (u-v)_+^{\frac{\mathfrak{S}_r +1}{2}}|^2 ~dx \right)^{\frac{N}{r'(N-2)}} \ \text{where} \ \Omega_1:= \{u \geq v\}
\end{split}
\]
which further implies
\[
\begin{split}
    \|\nabla (u-v)_+^{\frac{\mathfrak{S}_r+1}{2}}\|^2_{L^2(\Omega)} & \leq C \|f-g\|_{L^r(\Omega_1)}^{\frac{r(N-2)}{N-2r}} \quad \text{where} \quad C=  \left(\frac{(\mathfrak{S}_r+1)^2}{4 \mathfrak{S}_r} \right)^{\frac{r(N-2)}{N-2r}} \left(S(N)\right)^{\frac{N(r-1)}{N-2r}}
\end{split}
\]
and analogously
\[
\begin{split}
    \|\nabla (v-u)_+^{\frac{\mathfrak{S}_r+1}{2}}\|^2_{L^2(\Omega)} & \leq C \|f-g\|_{L^r(\Omega \setminus \Omega_1)}^{\frac{r(N-2)}{N-2r}}.
\end{split}
\]
\qed \vspace{0.1cm}\\
\textbf{Proof of Corollary \ref{result:9}} The proof follows from the inequality relation \eqref{cont:est6} by taking $f \leq g.$ \qed
\subsection{Proof of Theorem \ref{result:4}:} Let $\zeta+\gamma \leq 1$ and $u_n$ be the weak solution of the problem $(P_n)$ in the sense that
\begin{equation}\label{sense:approx:1}
\begin{split}
    \int_{\Omega} \nabla u_n \cdot \nabla \psi ~dx + \frac{C(N,s)}{2} \int_{\mathbb{R}^N} \int_{\mathbb{R}^N} &\frac{(u_n(x)- u_n(y))(\psi(x)- \psi(y))}{|x-y|^{N+2s}} ~dx ~dy \\
    &= \int_{\Omega} \frac{f_n(x)}{(u_n + \frac{1}{n})^\gamma} \psi ~dx \quad \forall \psi \in H_0^1(\Omega).
\end{split}
\end{equation}
From Lemma \ref{sobolev:reg} and Lemma \ref{bdry:est}, we know that, for any $\mathfrak{L}>0$, $u_n^{\frac{\mathfrak{L}+1}{2}}$ is uniformly bounded in $H_0^1(\Omega)$ and $u_n \geq C \delta(x)$ for $x \in \Omega$ and $C>0$ independent of $n$. Therefore, by taking $\mathfrak{L}=1$, we obtain
\begin{equation}\label{convergence:est:case:3}
\begin{split}
    & u_n \rightharpoonup u \ \text{in} \ H_0^1(\Omega), \ \ u_n \to u \ \text{in} \ L^{j}(\Omega) \ \text{for} \ 1 \leq j < 2^\ast\ \text{and a.e. in} \ \mathbb{R}^N.
\end{split}
\end{equation}
Using Hardy's inequality, for any $\psi \in H_0^1(\Omega)$, $\frac{f_n \psi}{(u_n+\frac{1}{n})^\gamma}$ is uniformly integrable. Indeed,
\begin{equation*}
    \begin{split}
        \int_{\Omega} \frac{f_n(x)}{(u_n + \frac{1}{n})^\gamma} \psi ~dx \leq \int_{\Omega} \frac{f(x)}{u_n^\gamma} \psi ~dx \lesssim \int_{\Omega} \delta^{1-\gamma-\zeta}(x) \frac{\psi}{\delta(x)} ~dx \lesssim \|\frac{\psi}{\delta}\|_{L^2(\Omega)} \leq \|\psi\|_{H_0^1(\Omega)}. 
    \end{split}
\end{equation*}
Finally, by using Vitali convergence theorem and convergence properties in \eqref{convergence:est:case:3}, we are able to pass limits in \eqref{sense:approx:1} and obtain
\[
\begin{split}
    \int_{\Omega} \nabla u\cdot \nabla \psi ~dx + \frac{C(N,s)}{2} \int_{\mathbb{R}^N} \int_{\mathbb{R}^N} &\frac{(u(x)- u(y))(\psi(x)- \psi(y))}{|x-y|^{N+2s}} ~dx ~dy = \int_{\Omega} \frac{f(x)}{u^\gamma} \psi ~dx \quad \forall \psi \in H_0^1(\Omega).
\end{split}
\]
Finally, by comparison principle (see \cite[Theorem 4.2]{Canino_et_al}) for any $n \in \mathbb{N}$, $u_n\leq v$ a.e. in $\Omega$ where $v$ is another weak solution of \eqref{main:problem}-\eqref{bdry:est}. Passing to the limit
$n \to \infty$ gives that u is a minimal solution.
\qed
\subsection{Proof of Theorem \ref{result:5}} Let $\zeta+\gamma > 1$ and $u_n$ be the weak solution of the problem $(P_n)$. From Lemma \ref{sobolev:reg} and Lemma \ref{Lem:apriori}, we have, for any $\mathfrak{L}>\mathfrak{L}^*$, $u_n^{\frac{\mathfrak{L}+1}{2}}$ is uniformly bounded in $H_0^1(\Omega)$ and $u_n \geq C(\omega)$ for $x \in \omega \Subset \Omega$ and $C>0$ independent of $n$. Now we divide the proof into two cases: \vspace{0.1cm}\\
\textbf{Case 1: } $\mathfrak{L}^* > 1$ \vspace{0.1cm}\\
In this case, $u_n$ is uniformly bounded in $H_{loc}^1(\Omega)$. Precisely,
\begin{equation}
\begin{split}
    \int_{\omega} |\nabla u_n|^2 ~dx \leq C^{-(\mathfrak{L}-1)}(\omega) \int_{\omega} u_n^{(\mathfrak{L}-1)}  |\nabla u_n|^2 ~dx \leq \frac{4 C^{-(\mathfrak{L}-1)}}{(\mathfrak{L} +1)} \int_{\omega} |\nabla u_n^{\frac{\mathfrak{L}+1}{2}}|^2 ~dx \leq C_0
\end{split}
\end{equation}
where $C_0$ is independent of $n.$ Then, there exists a $u \in H_{loc}^1(\Omega)$ such that
\begin{equation}\label{convergence:est:case:4}
\begin{split}
    & u_n \rightharpoonup u \ \text{in} \ H_{loc}^1(\Omega), \ \ u_n \to u \ \text{in} \ L^{j}_{loc}(\Omega) \ \text{for} \ 1 \leq j < 2^\ast\ \text{and a.e. in} \ \mathbb{R}^N.
\end{split}
\end{equation}
Using Hardy's inequality we have for any $\psi \in H_0^1(\Omega)$ with $\supp(\psi) \Subset \Omega$, $\frac{f_n \psi}{(u_n+\frac{1}{n})^\gamma}$ is uniformly integrable in $L^1(\omega)$. Then, by using Vitali convergence theorem, convergence properties in \eqref{convergence:est:case:4} and adopting the same arguments from \cite[Theorem 3.6]{Canino_et_al}, we are able to pass limits in \eqref{sense:approx:1}, {\it i.e.} for any $\psi \in \bigcup_{\omega \Subset \Omega} H_0^1(\omega)$, 
\[
\begin{split}
    \int_{\Omega} \nabla u\cdot \nabla \psi ~dx + \frac{C(N,s)}{2} \int_{\mathbb{R}^N} \int_{\mathbb{R}^N} &\frac{(u(x)- u(y))(\psi(x)- \psi(y))}{|x-y|^{N+2s}} ~dx ~dy = \int_{\Omega} \frac{f(x)}{u^\gamma} \psi ~dx.
\end{split}
\]
\textbf{Case 2: } $\mathfrak{L}^* \leq 1$ \vspace{0.1cm}\\
In this case, $u_n$ is uniformly bounded in $H_0^1(\Omega)$, precisely, by taking $\mathfrak{L}=1$. Now, by repeating the proof of Theorem \ref{result:4}, we obtain our claim.
\subsection{Proof of Theorem \ref{result:6}} From Theorem \ref{result:4}, Theorem \ref{result:5}, Lemma \ref{bdry:est} and using the fact that $u:=\lim_{n \to \infty} u_n$ is a weak solution of the main problem \eqref{main:problem}, we obtain that $u$ satisfies
\[
	C_1 \delta(x) \leq u(x) \leq C_2 \delta(x) \left\{
	\begin{array}{ll}
	1  & \text{ if } \  \zeta + \gamma < 1, \vspace{0.1cm}\\
	\ln\left(\frac{\mathcal{D}_\Omega}{\delta(x)}\right) & \text{ if } \ \zeta + \gamma = 1,
	\end{array} 
	\right.
\]
and 
\[
	C^1\delta^{\frac{2-\zeta}{\gamma+1}}(x)\leq u(x) \leq C^2 \delta^{\frac{2-\zeta}{\gamma+1}}(x).
\]
Now, it only remains to prove the optimal boundary behavior in case of $\zeta+ \gamma=1.$ Let $w_\Xi$ be the solution to the problem
\begin{equation}\label{log:problem}
\left\{
\begin{aligned}
-\Delta w_\Xi + (-\Delta)^s w_\Xi& = \frac{1}{\delta}\ln^{-\Xi}\left(\frac{\mathcal{D}_{\Omega}}{\delta}\right) , \quad w_\Xi>0 \quad &&\text{ in }  \Omega, \\
w_\Xi & = 0 &&  \text{ in } \ \mathbb{R}^N \setminus \Omega.
\end{aligned}
\right.
\end{equation}
where $\Xi \in [0,1).$ Then, by using the integral representation of the solution, Lemma \ref{lem:greest1} and Lemma \ref{lem:greest2}, we obtain 
	\begin{equation}\label{log:rep}
	\begin{aligned}
	w_\Xi(x):=  \left\{
	\begin{array}{ll}
	\mathbb{G}^\Omega\left[\frac{1}{\delta}\ln^{-\Xi}(\frac{\mathcal{D}_{\Omega}}{\delta}) \right](x) &  \text{ if } x \in \Omega, \\
	0
	& \text{ if } x \in \mathbb{R}^N \setminus \Omega.\\
	\end{array} 
	\right.
	\end{aligned}
	\end{equation}
and there exists $c_1, c_2>0$ such that
\begin{equation}\label{bodry:est}
   c_1 \delta(x) \ell^{1-\Xi}(\delta(x)) \leq w_\Xi(x) \leq c_2 \delta(x) \ell^{1-\Xi}(\delta(x)) \quad \forall \ x \in \Omega. 
\end{equation}
Moreover, for any $\Xi_0 <1$, $c_1$ and $c_2$ are uniform for any $0 \leq  \Xi \leq \Xi_0.$ Now we divide the proof into two cases: \vspace{0.1cm}\\
\textbf{Case 1:} $\zeta>0$\vspace{0.1cm}\\
Since $u$ satisfies \eqref{est:upper-lower:appr:solu:weak}, $\gamma+ \zeta =1$ and $f \in \mathcal{A}_\zeta$, there exists $d_0, d_1>0$ such that $d_0 \delta^{-\zeta}(x) \leq f(x) \leq d_1 \delta^{-\zeta}(x)$ for $x \in \Omega$ and $u$ satisfies
\[ (-\Delta) (d_0 w_\gamma) + (-\Delta)^s (d_0 w_\gamma) = d_0 \delta^{-1}(x) \ln^{-\gamma}\left(\frac{\mathcal{D}_\Omega}{\delta(x)}\right)  \leq \frac{f(x)}{u^\gamma} = (-\Delta) u + (-\Delta)^s u.\]
Now, by using the comparison principle and \eqref{bodry:est} with $\Xi_1:=\gamma$, we obtain
\begin{equation*}
c_1 d_0 \delta(x) \ln^{1-\Xi_1}\left(\frac{\mathcal{D}_{\Omega}}{\delta(x)}\right) \leq d_0 w_{\Xi_1}(x) \leq u(x) \ \text{for} \ x \in \Omega   
\end{equation*}
and 
\[(-\Delta) u + (-\Delta)^s u = \frac{f(x)}{u^\gamma} \leq  \frac{d_1}{(c_1 d_0)^\gamma} \delta^{-1}(x) \ln^{-\Xi_2}\left(\frac{\mathcal{D}_\Omega}{\delta(x)}\right) = (-\Delta) \left(\frac{d_1}{(c_1 d_0)^\gamma} w_{\Xi_2}\right) + (-\Delta)^s \left(\frac{d_1}{(c_1 d_0)^\gamma} w_{\Xi_2}\right)\]
where $\Xi_2:= \gamma(1-\Xi_1).$ Again, using the comparison principle, we obtain
\begin{equation*}
u(x) \leq \frac{d_1}{(c_1 d_0)^\gamma} w_{\Xi_2}(x) \leq \frac{d_1 c_2}{(c_1 d_0)^\gamma} \delta(x) \ln^{1-\Xi_2}\left(\frac{\mathcal{D}_{\Omega}}{\delta(x)}\right) \ \text{for} \ x \in \Omega   
\end{equation*}
Iterating these estimates, we obtain for any $j \in \mathbb{N}$
\[
    \begin{split}
    \frac{(c_1 d_0)^{1+\gamma^2 + \dots + \gamma^{2j}}}{(d_1c_2)^{\gamma + \gamma^3 + \dots + \gamma^{2j-1}}} \delta(x) &\ln^{1-\gamma + \gamma^2-\gamma^3 \dots -\gamma^{2j+1} }\left(\frac{\mathcal{D}_\Omega}{\delta(x)}\right) \leq u(x) \\
    &\leq  \frac{(c_1 d_0)^{1+\gamma^2 + \dots + \gamma^{2j}}}{(d_1c_2)^{\gamma + \gamma^3 + \dots + \gamma^{2j+1}}} \delta(x) \ln^{1-\gamma + \gamma^2 \dots + \gamma^{2j+2} }\left(\frac{\mathcal{D}_\Omega}{\delta(x)}\right) 
    \end{split}
\]
Passing to the limit $j \to \infty$, we obtain
\[ d_2 \delta(x) \ln^{\frac{1}{1+\gamma}}\left(\frac{\mathcal{D}_\Omega}{\delta(x)}\right) \leq u(x) \leq d_3 \delta(x) \ln^{\frac{1}{1+\gamma}}\left(\frac{\mathcal{D}_\Omega}{\delta(x)}\right)
\]
where $d_2, d_3>0$ are constants depending upon $\gamma, d_0, d_1, c_1$ and $c_2.$ 
\vspace{0.1cm}\\
\textbf{Case 2:} $\zeta=0$\vspace{0.1cm}\\
In this case, by taking $\Xi=\frac{1}{2}$ in \eqref{log:problem} and \eqref{log:rep}, there exists $c_3, c_4>0$ such that
\begin{equation}\label{bodry:est:1}
   c_3 \delta(x) \ln^{\frac{1}{2}}\left(\frac{\mathcal{D}_{\Omega}}{\delta(x)}\right) \leq w_\frac{1}{2}(x) \leq c_4 \delta(x) \ln^{\frac{1}{2}}\left(\frac{\mathcal{D}_{\Omega}}{\delta(x)}\right) \quad \forall \ x \in \Omega.
\end{equation}
and for a positive constant $C$ large enough, we have
\[
(-\Delta) (C w_\frac{1}{2}) + (-\Delta)^s (C w_\frac{1}{2}) \geq \frac{1}{C w_\frac{1}{2}} \ \text{and} \ (-\Delta) \left(\frac{w_\frac{1}{2}}{C}\right) + (-\Delta)^s \left(\frac{w_\frac{1}{2}}{C}\right) \leq \frac{C}{ w_\frac{1}{2}}.
\]
Therefore again by comparison principle, we obtain \[ \frac{c_3}{C} \delta(x) \ln^{\frac{1}{2}}\left(\frac{\mathcal{D}_{\Omega}}{\delta(x)}\right) \leq u(x) \leq \frac{c_4}{C} \delta(x) \ln^{\frac{1}{2}}\left(\frac{\mathcal{D}_{\Omega}}{\delta(x)}\right) \quad \forall \ x \in \Omega.\]
The only if statement in proving optimal Sobolev regularity follows from the Hardy inequality and the boundary behavior of the weak solution. Precisely, if \[ \mathfrak{L} \leq \left\{
	\begin{array}{ll}
	0  & \text{ if } \  \zeta+ \gamma \leq  1, \vspace{0.1cm}\\
	\mathfrak{L}^* & \text{ if } \ \zeta+ \gamma > 1.
	\end{array} 
	\right. \quad \text{then} \quad u^{\frac{\mathfrak{L}+1}{2}} \notin H_0^{1}(\Omega).\] Indeed, in case of $\zeta+ \gamma \leq  1$ and $\mathfrak{L} \leq 0$, we have
$$ \|u^\frac{\mathfrak{L}+1}{2}\|_{H_0^1(\Omega)}^2 \geq C\int_{\Omega} \frac{|u^\frac{\mathfrak{L}+1}{2}(x)|^2}{|d(x)|^2} ~dx \geq C\int_{\Omega} d^{\mathfrak{L}-1}(x) ~dx =\infty.$$
In the same way, if $\zeta+ \gamma >1$ and $\mathfrak{L} \leq \mathfrak{L}^*$, then 
$$ \|u^\frac{\mathfrak{L}+1}{2}\|_{H_0^1(\Omega)}^2  \geq C\int_{\Omega} \frac{|u^\frac{\mathfrak{L}+1}{2}(x)|^2}{|d(x)|^2} ~dx \geq C\int_{\Omega} d^{\frac{(\mathfrak{L}+1)(2-\zeta)}{(1+\gamma)}-2}(x) ~dx =\infty$$
and we deduce $u^\frac{\mathfrak{L}+1}{2} \notin H_0^{1}(\Omega).$ \qed 
\subsection{Proof of Theorem \ref{result:7}} Let $\zeta \geq 2$ and $f \in \mathcal{A}_\zeta$. We choose $\Gamma \in (0,1)$ and  $\zeta_0 <2$ such that $g(x) \leq f(x)$ for $g \in \mathcal{A}_{\zeta_0}$ and the constant $\Gamma$ is independent of $\zeta_0$ for $\zeta_0 \geq \zeta_0^*$ with $\zeta_0^*>0$. To prove our claim, we proceed by contradiction. Assume there exist a weak solution $w \in H_{loc}^{1}(\Omega)$ of the problem \eqref{main:problem} and $\mathfrak{L}_0 \geq 1$ such that $w^{\mathfrak{L}_0} \in H_0^{1}(\Omega)$.\\
For $n \in \mathbb{N}$, let $v_n \in H_0^{1}(\Omega) \cap C^{0,\ell}(\overline{\Omega})$ be the unique weak solution of
\begin{equation}\label{bv}
\int_{\Omega} \nabla v_n \cdot \nabla  \phi ~dx + \frac{C(N,s)}{2} \iint_{\mathbb{R}^{2N}} \frac{(v_n(x)- v_n(y)) (\phi(x)- \phi(y))}{|x-y|^{N+2s}} ~dx ~dy = \int_{\Omega} \frac{\Gamma g_n(x)}{(v_n+ \frac{1}{n})^\gamma} \phi ~dx 
\end{equation}
for any $\phi \in H_0^{1}(\Omega)$. By the continuity of $v_n$, for given $\theta >0$, there exists a $\eta {= \eta(n,\theta)} >0$ such that $v_n \leq  \frac{\theta}{2}$ in $\Omega_\eta$. Since $w \geq 0$, then $ u:= v_n- w- \theta \leq - \frac{\theta}{2} <0$ in $\Omega_\eta$ and 
\begin{equation*}\label{supp}
\mbox{supp}({u}^+) \subset \mbox{supp}((v_n- \theta)^+) \subset \Omega \setminus \Omega_\eta.
\end{equation*}
We have $u^+\in H_0^{1}(\tilde \Omega)\subset H_0^{1}(\Omega)$ for some $\tilde \Omega$ such that $\Omega\setminus \Omega_\eta \subset \tilde \Omega \Subset \Omega$. Hence, choosing $u^+$ as a test function in \eqref{bv}, we get
\begin{equation}\label{bv4}
\begin{split}
\int_{\Omega} \nabla v_n \cdot \nabla u^+ ~dx &+ \frac{C(N,s)}{2} \iint_{\mathbb{R}^{2N}} \frac{(v_n(x)- v_n(y)) (u^+(x)- u^+(y))}{|x-y|^{N+2s}} ~dx ~dy \\
&= \int_{\Omega} \frac{\Gamma g_n(x)}{(v_n+ \epsilon)^\gamma} u^+ ~dx \leq \int_{\Omega} \frac{\Gamma g_n(x)}{v_n^\gamma} u^+ ~dx. 
\end{split}
\end{equation} 
Moreover, $w$ is a weak solution of $(P)$ and taking $u^+\in H_0^{1}(\tilde \Omega)$ as test function, we have
\begin{equation}\label{bv2}
\begin{split}
\int_{\Omega} \nabla w \cdot \nabla u^+ ~dx &+ \frac{C(N,s)}{2} \iint_{\mathbb{R}^{2N}} \frac{(w(x)- w(y)) (u^+(x)- u^+(y))}{|x-y|^{N+2s}} ~dx ~dy \\
&= \int_{\Omega} \frac{f(x)}{w^\gamma} w^+ ~dx \geq \int_{\Omega} \frac{ \Gamma g_n(x)}{w^\gamma} w^+ ~dx. 
\end{split}
\end{equation}
By subtracting \eqref{bv2} and \eqref{bv4}, we get
\[
\begin{split}
0 \leq \int_{\Omega} (\nabla v_n - \nabla w) \cdot \nabla u^+ ~dx &+ \frac{C(N,s)}{2} \iint_{\mathbb{R}^{2N}} \frac{((v_n(x)- v_n(y))-(w(x)- w(y))) (u^+(x)- u^+(y))}{|x-y|^{N+2s}} ~dx ~dy\\
& \leq \int_{\Omega} \left(\frac{\Gamma g_n(x)}{v_n^\gamma} - \frac{\Gamma g_n(x)}{w^\gamma} \right) u^+ ~dx  \leq 0.
\end{split}
\]
which further implies $w^+= (v_n- w -\theta)^+=0$ a.e. in $\Omega$. Since  $\theta$ is arbitrary, we deduce $v_n \leq w$ in $\Omega.$  Using Lemma \ref{bdry:est}, we obtain
$$c_1 (\delta(x) + n^{\frac{-(1+\gamma)}{(2-\zeta_0)}})^{\frac{2-\zeta_0}{\gamma+1}} - c_2 n^{-1} \leq v_n \leq w \ \text{in} \ \Omega.$$ 
Now, by using the Hardy inequality and $w^{\mathfrak{L}_0} \in H_0^{1}(\Omega)$, we obtain
$$
\int_{\Omega} \left(c_1 (\delta(x) + n^{\frac{-(1+\gamma)}{(2-\zeta_0)}})^{\frac{2-\zeta_0}{\gamma+1}} - c_2 n^{-1}\right)^{2\mathfrak{L}_0} d^{-2}(x) ~dx \leq \int_{\Omega} \left|\frac{w^{\mathfrak{L}_0}}{d(x)}\right|^2 ~dx < \infty.
$$
Now, by choosing $\zeta_0$ close enough to $2$ and by taking $n \to \infty$, we obtain that the left hand side is not finite, which is a contradiction and hence claim.  \qed 
\section{Appendix}\label{sec:append}
In this part, we recall some preliminary results comprised of upper and lower estimates of Green kernel, lower Hopf type estimate, and action of the Green kernel on the distance power functions.  
\begin{pro}[Theorem 1, \cite{Chenkim_Green1}]
Let $\Omega$ be a $C^{1,1}$ open set in $\mathbb{R}^N$ and $\mathcal{G}^\Omega$ denote the Green kernel for the mixed operator $(-\Delta) + (\Delta)^s$ defined on $\dom(\mathcal{G}^{\Omega}) := \Omega \times \Omega \setminus \{(x,x): x \in \Omega\}$. Then, for any $(x,y) \in \dom(\mathcal{G}^{\Omega}):$
	\begin{equation*}
	\mathcal{G}^{\Omega}(x,y) \asymp \frac{1}{|x-y|^{N-2}} \left( \frac{\delta(y) \delta(x)}{|x-y|^2} \wedge 1 \right) \ \text{ if } \ N \geq 3.
	\end{equation*}
\end{pro}
\begin{Lem}[Theorem 2.6, \cite{AbaGomVaz_2019}]\label{Hopf}
There exists $C>0$ such that for any $f \geq 0$, 
	$$\mathbb{G}^{\Omega}\left[f\right](x):= \int_{\Omega} \mathcal{G}^{\Omega}(x,y) f(y) ~dy \geq C \delta(x)^{\gamma} \int_\Omega \delta^\gamma(y) f(y) ~dy, \ x \in \Omega.$$
\end{Lem}
\begin{Lem}[Theorem 3.4, \cite{AbaGomVaz_2019}]\label{est:green}
	Assume $\beta < 2$. Then $\delta^{-\beta} \in L^1(\Omega, \delta)$ and $ \mathbb{G}^\Omega[\delta^{-\beta}] \asymp \delta^\vartheta$ in $\Omega$
	where 
	\begin{equation}\label{bdry:beha}
	\begin{aligned}
	\vartheta=  \left\{
	\begin{array}{ll}
	1 &  \text{ if } \beta<1, \\
	1 \, (\text{and logarithmic weight}) & \text{ if } \beta=1,\\
	2-\beta & \text{ if } \beta>1 .\\
	\end{array} 
	\right.
	\end{aligned}
	\end{equation}
By logarithmic weight we mean $  \mathbb{G}^\Omega[\delta^{-\beta}] \asymp \delta \ln(\frac{\mathcal{D}_\Omega}{\delta}).$
\end{Lem}
\noindent We also recall some technical algebraic inequality from \cite[Lemma 2.22]{Abdel} and \cite{Stampacchia}.
\begin{Lem}\label{prelim:inequa}
\begin{itemize}
    \item[(i)] For any $x, y \geq 0$ and $\theta >0$, 
    $$\frac{4 \theta}{(\theta +1)^2} \left( x^{\frac{\theta +1}{2}} - y^{\frac{\theta +1}{2}} \right) \leq (x-y) \left(x^\theta - y^\theta \right).$$
    \item[(ii)] Let $0< \theta \leq 1$, then for every $x, y \geq 0$
    $$|x^\theta - y^\theta| \leq |x-y|^\theta.$$
\end{itemize}
\end{Lem}
\begin{Lem}\label{lem:iter}
Let $\psi : \mathbb{R}^+ \to \mathbb{R}^+$ be a non-increasing function such that
$$\psi(h) \leq \frac{C \psi(k)^\delta}{(h-k)^\eta}\ \text{for all} \ j>k>0,$$
where $C>0$, $\delta>1$ and $\eta >0.$ Then $\psi(d)=0$, where $d^\eta= C \psi(0)^{\delta-1} 2^{\frac{\delta\eta}{\delta-1}}$
\end{Lem}

\section*{Acknowledgements} Rakesh Arora acknowledges the support of the Research Grant from Czech Science Foundation, project GJ19-14413Y for this work.
The research of Vicen\c tiu D. R\u adulescu  was supported by a grant of the Romanian Ministry of Research, Innovation and Digitization, CNCS/CCCDI--UEFISCDI, project number PCE 137/2021, within PNCDI III.

\end{document}